\documentclass[DIV12,parskip=half]{scrartcl}

\usepackage[utf8]{inputenc}
\usepackage[english]{babel}

\usepackage{amssymb,amsmath,amsfonts}
\usepackage{amsthm}
\usepackage[centercolon]{mathtools}
\usepackage{nicefrac}

\usepackage[T1]{fontenc}
\usepackage[proportional]{libertine}
\usepackage[libertine,liby,vvarbb]{newtxmath}
\usepackage[scaled=0.95,varqu,varl]{inconsolata}
\useosf
\usepackage[kerning,spacing]{microtype}

\numberwithin{equation}{section}
\usepackage{enumerate}
\usepackage[colorlinks,linkcolor=blue,citecolor=blue,urlcolor=blue,unicode]{hyperref}
\usepackage[nameinlink,capitalize]{cleveref}

\newtheorem{theorem}{Theorem}[section]
\newtheorem{proposition}[theorem]{Proposition}
\newtheorem{definition}[theorem]{Definition}
\newtheorem{lemma}[theorem]{Lemma}
\newtheorem{remark}[theorem]{Remark}
\newtheorem{example}[theorem]{Example}
\newtheorem{corollary}[theorem]{Corollary}
\newtheorem{assumption}[theorem]{Assumption}

\newcommand{\R}{\mathbb{R}}
\newcommand{\N}{\mathbb{N}}

\newcommand{\norm}[1]{\ensuremath{\lVert#1\rVert}}
\newcommand{\abs}[1]{\ensuremath{\lvert#1\rvert}}
\newcommand{\sca}[3][]{\langle #2\,,\,#3\rangle}
\newcommand{\Span}[1]{\langle {#1} \rangle}
\newcommand{\of}[2][1]{\left(#2\right)}
\newcommand{\Lorl}[1][\Phi]{L^{#1}}
\newcommand{\LlogL}{\ensuremath{L\log L}}
\newcommand{\Lexp}{\ensuremath{L_{\mathrm{exp}}}}
\newcommand{\M}{\ensuremath{\mathfrak{M}}}
\newcommand{\Mp}{\ensuremath{\mathfrak{M}_+}}
\renewcommand{\P}{\ensuremath{\mathcal{P}}}

\newcommand{\CC}{\ensuremath{\mathcal{C}}}
\newcommand{\CCb}{\ensuremath{\CC_{\mathrm{b}}}}

\newcommand{\test}{\CC_{\mathrm{c}}^{\infty}}
\newcommand{\bd}{\ensuremath{\mathop{}\!\mathrm{d}}}
\newcommand{\leb}{\mathcal{L}}
\newcommand{\dleb}{\bd\leb}
\newcommand{\dlambda}{\bd\lambda}
\newcommand{\hlambda}{\hat\lambda}
\newcommand{\dhlambda}{\bd\hlambda}

\newcommand{\ds}{\bd s}

\newcommand{\dw}{\bd w}
\newcommand{\dpi}{\bd \pi}

\newcommand{\dmu}{\bd\mu}
\newcommand{\dnu}{\bd\nu}
\newcommand{\intO}[1][]{\int_{\Omega_{#1}}}

\newcommand{\wsto}{\xrightharpoonup{*}}
\newcommand{\phito}{\xrightarrow{\Phi}}
\DeclareMathOperator*\Gammalim{\begingroup\mathgroup-1 \Gamma\endgroup-lim}

\DeclareMathOperator\spt{spt}
\newcommand{\pos}[1]{{#1}_+}
\newcommand{\negprt}[1]{{#1}_-}
\newcommand{\blank}{\, \boldsymbol{.}\, }
\newcommand{\set}[2]{\ensuremath{\left\{#1\,\middle|\,#2\right\}}}
\usepackage{dsfont}
\newcommand{\1}{\mathds{1}} 
\newcommand{\proj}[1]{P_{#1}}
\newcommand{\ppfw}[1][i]{(\proj{#1})_\#}
\newcommand{\calV}{\mathcal{V}}
\newcommand{\ext}[1]{{#1}_\infty}

\crefname{assumption}{Assumption}{Assumptions}

\begin{document}
	\title{Orlicz space regularization of continuous optimal transport problems}
	\author{
		Dirk Lorenz
		\thanks{%
			Institute of Analysis and Algebra,
			TU Braunschweig,
			38092 Braunschweig, Germany,
			(\texttt{d.lorenz@tu-braunschweig.de, h.mahler@tu-braunschweig.de})
		}
		\and
		Hinrich Mahler
		\footnotemark[1]
	}

	\maketitle
	
	\begin{abstract}
		In this work we analyze regularized optimal transport problems in the so-called Kantorovich form,
		i.e. given two Radon measures on two compact sets, the aim is to find a transport plan, which is another Radon measure on the product of the sets, that has these two measures as marginals and minimizes the sum of a certain linear cost function and a regularization term.
		We focus on regularization terms where a Young's function applied to the (density of the) transport plan is integrated against a product measure. This forces the transport plan to belong to a certain Orlicz space.
		The predual problem is derived and proofs for strong duality and existence of primal solutions of the regularized problem are presented. Existence of (pre-)dual solutions is shown for the special case of $L^p$-regularization for $p\geq 2$.
		Moreover, two results regarding $\Gamma$-convergence are stated: The first is concerned with marginals that do not lie in the appropriate Orlicz space and guarantees $\Gamma$-convergence to the original Kantorovich problem, when smoothing the marginals.
		The second result gives convergence of a regularized and discretized problem to the unregularized, continuous problem.	
	\end{abstract}

	\section{Introduction}
	
		In this paper we consider the optimal transport problem in the Kantorovich form in the following setting: For compact sets $\Omega_1,\,\Omega_2\subset \R^n$, probability measures $\mu_1,\mu_2$ on $\Omega_1,\Omega_2$, respectively, and a real-valued continuous cost function $c:\Omega_1\times\Omega_2\to \R$ we want to solve
		\begin{equation}\tag{OT}\label{eq:ot}
			\inf_{\pi}\int_{\Omega_1\times\Omega_2}c\dpi
		\end{equation}
		where the infimum is taken over all probability measures on $\Omega_1\times \Omega_2$ which have $\mu_1$ and $\mu_2$ as their first and second marginals, respectively.
		This problem has been well studied, and an overview is given in the recent books \cite{santambrogio:2015,villani:2008,peyre:2019}.
		For example, it is known that the problem has a solution $\bar\pi$ and that the support of $\bar\pi$ is contained in the so-called $c$-superdifferential of a $c$-concave function on $\Omega_1$, see \cite[Theorem~1.13]{ambrosio:2013}. In the case where $c(x_1,x_2) = |x_1-x_2|^2$ is the squared Euclidean distance, this implies that the support of an optimal plan $\pi$ is singular with respect to the Lebesgue measure.
		This motivates the use of regularization of the continuous problem to obtain approximate solutions that are absolutely continuous w.r.t. given measures. That in turn allows to apply classical discretization techniques to solve the regularized problem approximately.
		
		A regularization method that received much attention recently is regularization with the negative entropy of $\pi$, i.e. adding a term $\int_{\Omega_{1}\times\Omega_{2}}\Phi(\pi)\dlambda$ with $\Phi(t) = t\log(t)$ and some measure $\lambda$ on $\Omega_{1}\times \Omega_{2}$ \cite{carlier:2017,cuturi:2013,cuturi:2015,clason:2021,benamou:2015}. Since $\pi$ is a measure, one has to interpret $\Phi(\pi)$ appropriately: One should think of $\pi$ as the Radon-Nikodym derivative of $\pi$ with respect to the regularization measure $\lambda$, and we will make this distinction explicit in the following.
		In~\cite{clason:2021} entropic regularization with respect to the Lebesgue measure is considered and it is shown that the analysis of entropically regularized optimal transport problems naturally takes place in the function space $\LlogL$ (also called Zygmund space~\cite{bennett:1988}) and that optimal plans for entropic regularization are always in $\LlogL(\Omega_{1}\times \Omega_{2})$ and exist if and only if the marginals are in the spaces $\LlogL(\Omega_{i})$. These spaces are an example of so-called Orlicz spaces~\cite{rao:1988}.
		This motivates the analysis of regularization in arbitrary Orlicz spaces in this paper. Another motivation to study a more general regularization comes from the fact that regularization with the $L^{2}$-norm has been shown to be beneficial in some applications, see~\cite{roberts:2017,blondel:2017,dessein:2018,lorenz:2019}.
        Using the product of the marginals $\lambda = \mu_{1}\times \mu_{2}$ for regularization has been considered in the case of entropic regularization~\cite{genevay:2019entropy,peyre:2019,vialard:2019}. In this case one can show existence of the dual problem with different techniques. These observations motivate us to consider regularization with Young's function with respect to general measures. Notable regularizations that our approach covers are $L^p$ regularization with $p>1$ arbitrary and the Tsallis entropy \cite{muzellec:2017}.
        
        The notion of Orlicz spaces in the context of convex integral functionals has previously been used in \cite{Leonard:2008}, where the author considers a more general setting than the one presented here. More precisely, the spaces used in \cite{Leonard:2008} are a generalization of the Orlicz spaces used here, which are also known as Musielak-Orlicz spaces \cite{musielak:1983}. Existence of both primal and dual optimizers are covered.
        By choosing $\gamma^*(z,t) = \varepsilon\cdot\Phi(t) + c(z)t + A(z)$ with $A(z) :=\min_t \varepsilon\cdot\Phi(t) + c(z)t$ and regularization parameter $\varepsilon$, a problem similar to the one considered here is recovered. The difference lies in the fact that the cost function $c$ is part of the definition of the relevant Musielak-Orlicz spaces in this case and hence, the analysis takes place in different spaces.
        As the aim of \cite{Leonard:2008} is to weaken the necessary assumptions as much as possible, the overall setting is more abstract and the proofs rely heavily on the author's work \cite{Leonard:2010}. Here we aim for a self-contained, more elementary treatment of the problem.

        During the review process of the present work, we became aware of the preprint \cite{dimarino:2020}, which also considers regularization of the Kantorovich problem with Young's functions. While the underlying domains $\Omega_i$ are chosen to be general complete separable metric spaces, only regularization with respect to the marginal measures is considered and this allows the authors to derive existence of dual solutions independent of the form of the regularization.

        Moreover, in \cite{paty:2020} regularization with general convex functionals $F:\M(\Omega)\to\R\cup\{\infty\}$ is considered. A duality result similar to our \cref{thm:str:dual} is derived and existence of primal solutions is covered. However, \cite[Theorem 2]{paty:2020} implies the existence of \emph{continuous} dual optimizers. As \cref{example:no-cont-duals} below demonstrates, this can not be the case in general.
		
	\subsection{Notation and problem statement}
	
		Let us first fix some notation before we formulate our problem.
		The spaces of Radon and probability measures on $\Omega\subset\R^n$ will be denoted by $\M(\Omega)$ and $\P(\Omega)$, respectively. The cone of non-negative Radon measures will be denoted by $\Mp(\Omega)$.
		With $\CC(\Omega)$ and $\CCb(\Omega)$ we denote the spaces of continuous functions and bounded, continuous functions, respectively.
		The Lebesgue measure will be denoted by $\leb$ and integrals w.r.t. the Lebesgue measure are simply denoted by $\bd x$ with the appropriate integration variable $x$.
		
        In the following we will consider compact domains $\Omega_{1}$, $\Omega_{2}$ equipped with finite measures  $\lambda_1\in\Mp(\Omega_1)$ and $\lambda_2\in\Mp(\Omega_2)$, respectively. The measures $\lambda_1$ and $\lambda_2$ will be assumed to have full support, i.e. $\spt \lambda_i = \Omega_i$, for $i=1,2$. We will denote $\Omega := \Omega_{1}\times\Omega_{2}$ and $\lambda := \lambda_1\otimes\lambda_2$.
		For the space of $p$-integrable functions on $\Omega$ with respect to the measure $\nu$, the symbol $L^p(\Omega, \bd\nu)$ will be used.
        When a measure $\nu$ is absolutely continuous with respect to another measure $\mu$, written as $\nu\ll\mu$, the Radon-Nikodym derivative of $\nu$ w.r.t. to $\mu$, i.e. the density of $\nu$ w.r.t $\mu$, will be denoted by $\tfrac\dnu\dmu$.

		The characteristic function of a set $A$ will be denoted by $\1_A$. For two functions $f:\Omega_1\to\R$ and $g:\Omega_2\to\R$, denote by $f\oplus g:\Omega_1\times\Omega_2\to\R$, $(x_1, x_2)\mapsto f(x_1) + g(x_2)$ the outer sum of $f$ and $g$. This notation generalizes to measures $\nu_1$, $\nu_2$ on $\Omega_1$, $\Omega_2$, respectively, by $\nu_1\oplus\nu_2 := \nu_1\otimes\leb + \leb\otimes\nu_2$.
		For $\nu\in\M(\Omega_1)$ and $f: \Omega_1\to\Omega_2$, the pushforward of $\nu$ by $f$ will be denoted as $f_\#\nu$, i.e. the measure on $\Omega_2$ defined by $f_\#\nu(A) := \nu(f^{-1}(A))$. Most importantly, the pushforward of the coordinate projections $\proj{i}: \Omega_1\times\Omega_2\to\Omega_i$, $\proj{i}(x_1,x_2) = x_i$ will be used. Note, that $\ppfw \pi$ is the $i$-th marginal of $\pi\in\M(\Omega_1\times\Omega_2)$.
                For real-valued functions $f$ we denote by $f_{+} := \max(f,0)$ the positive part and by $f_{-}:= -\min(f,0)$ the negative part.
		Finally, for a function $g:[0,\infty)\to \R\cup\{\infty\}$ denote by $\ext g$ its extension to the real line by infinity, i.e.
		\[
			\ext g(x) := \begin{cases}
				g(x), &x\geq 0,\\
				\infty, &\text{else.}
			\end{cases}
		\]
		
		The Orlicz space regularized Kantorovich problem of optimal transport considered in this work now reads as
		\begin{equation}\tag{P}\label{eq:reg_kantorovich}
			\inf_{\substack{\pi\in\P(\Omega),\,\pi\ll\lambda\\\ppfw[i]{\pi} = {\mu_i},\,i=1,2}} \intO c \dpi + \gamma \intO \ext\Phi(\tfrac{\dpi}{\dlambda}) \dlambda\,,
		\end{equation}%
		where $\Phi$ is a so-called \emph{Young's function}.
		Note that the regularization of $\pi$ is employed w.r.t. some product measure $\lambda$. Important cases are $\lambda = \leb$ and  $\lambda = \mu_1\otimes\mu_2$.
		Note also that $\pi$ is required to be absolutely continuous w.r.t. $\lambda$. This is due to the fact that even for Young's functions $\Phi$ satisfying modest conditions like $\frac{\Phi(t)}{t} \not\to\infty$, as $t\to\infty$, by e.g. \cite[Theorem 5.19]{fonseca:2007} the optimal $\pi$ may have a singular part w.r.t. $\lambda$. That however, would make the process of regularizing futile.
		Therefore we will require $\lim_{t\to\infty}\nicefrac{\Phi(t)}{t} =\infty$ throughout the paper.
                
                Now consider the marginal constraints. As we will see later in \cref{thm:proj_contraction,thm:primal_existence}, it is necessary that the marginals $\mu_{i}$ are absolutely continuous with respect to $\lambda_{i}$ and that $\Phi(\tfrac{\dmu_{i}}{\dlambda_{i}})$ is integrable with respect to $\lambda_{i}$. To formulate the marginal constraints in terms of the densities, we recall that  $\ppfw[1]{\pi} = {\mu_1}$ if for all $\lambda_{1}$-measurable sets $A$ it holds that $\pi(A\times\Omega_{2}) = \mu_{1}(A)$. In terms of densities and integrals, this reads as
                \[
                	\int_{\Omega_{2}}\int_{A}\tfrac{\dpi}{\dlambda}\dlambda_{1}\dlambda_{2} = \int_{A}d\mu_{1} = \int_{A}\tfrac{\dmu_{1}}{\dlambda_{1}}\dlambda_{1}.
                \]
                Using Fubini's theorem we get
                \[
                	\int_{A}\left(\int_{\Omega_2} \tfrac{\dpi}{\dlambda}\dlambda_{2} -\tfrac{\dmu_{1}}{\dlambda_{1}}\right)\dlambda_{1} = 0\,,
                \]
                and hence, the marginal constraints read as
                \[
                	\int_{\Omega_{2}} \tfrac{\dpi}{\dlambda}\dlambda_{2} =\tfrac{\dmu_{1}}{\dlambda_{1}}\quad\text{$\lambda_1$-a.e.}\quad\text{and}\quad\int_{\Omega_{1}} \tfrac{\dpi}{\dlambda}\dlambda_{1} =\tfrac{\dmu_{2}}{\dlambda_{2}}\quad\text{$\lambda_2$-a.e.}
                \]
                
		Note that for the integral $\intO c\dpi = \intO c\tfrac{\dpi}{\dlambda}\dlambda$ to exist, the cost function $c$ does not need to be continuous and the problem may be formulated for more general cost functions. However, some of the results in this work require $c$ to be continuous and for simplicity this shall be assumed throughout the paper.

        Let us summarize our assumptions:

        \begin{assumption}\label{assmpt:general}
            For $i=1,2$ let $\Omega_{i}\subset\R^n$ be compact domains equipped with finite measures  $\lambda_i\in\Mp(\Omega_i)$ with $\spt\lambda_i = \Omega_i$. Let $c:\Omega\to\R$ be a continuous cost function. Let $\Phi$ be such that $\lim_{t\to\infty}\nicefrac{\Phi(t)}{t} =\infty$. Finally, for $i=1,2$ let $\mu_i\in\P(\Omega_i)$ such that $\mu_i\ll\lambda_i$ and $\Phi(\tfrac{\dmu_{i}}{\dlambda_{i}})$ is integrable with respect to $\lambda_{i}$.
        \end{assumption}

	\subsection{Contribution and Organization}
	
		The notions of Young's functions and Orlicz spaces are introduced in \cref{sec:yf_os} alongside some auxiliary results that will be used in the later sections. \Cref{sec:existence} deals with the question of existence of solutions in the framework of Fenchel duality.
		The first contribution (\cref{thm:primal_existence}) guarantees existence of solutions of problem \eqref{eq:reg_kantorovich}, which generalizes the corresponding results of \cite{clason:2021,lorenz:2019}. Afterwards, the predual problem is analyzed for the special case $\Phi(t) = \nicefrac{t^p}p$ for $p>1$. As second contribution \cref{thm:dual:lq} gives existence of dual optimizers in $L^q$, where $\nicefrac 1p + \nicefrac 1q = 1$ and $p\geq 2$. This generalizes the corresponding result of \cite{lorenz:2019}.
		In \cref{sec:gamma_conv} $\Gamma$-convergence of different related problems is considered. First, a continuous, regularized problem with arbitrary marginals is considered. \Cref{thm:conv_convergence} extends \cite[Theorem 5.1]{clason:2021} and guarantees $\Gamma$-convergence to the unregularized problem \eqref{eq:ot} when smoothing the marginals. Note that the case of $\Gamma$-convergence for fixed marginals in $\LlogL(\Omega)$ has been treated in \cite[Theorem~2.7]{carlier:2017} for $\Omega = \R^{n_1}\times \R^{n_2}$. While \cref{thm:conv_convergence} is stated only for compact $\Omega$, it allows for marginals not in $\LlogL(\Omega)$ and a coupled reduction of the regularization and the smoothing parameter.
		The final contribution (\cref{thm:disc:conv}) is the proof of $\Gamma$-convergence of a discretized and regularized optimal transport problem to the unregularized continuous problem \eqref{eq:ot}. The result covers both entropic and quadratic regularization as special cases.
        Some of the results in this paper are a direct generalization of results of previous papers and their proofs also follow the general proof strategy.
	
	\section{Young's functions and Orlicz spaces}\label{sec:yf_os}
		\label{sec:orlicz_spaces}\label{sec:young_functions}
		
		In this section, some notions about Young's functions and Orlicz spaces are introduced. For a more detailed introduction, see~\cite{bennett:1988,rao:1988}.
	
		\begin{definition}[Young's function
			{\cite[Definitions IV.8.1, IV.8.11]{bennett:1988}}%
		]\label{def:yf}
			\begin{enumerate}[i)]
				\item Let $\varphi : [0, \infty) \to [0, \infty]$ be increasing and lower semi-continuous,	with $\varphi(0) = 0$. Suppose that $\varphi$ is neither identically zero nor identically infinite on $(0,\infty)$. Then the function $\Phi$ defined by $\Phi(t) := \int_0^t \varphi(s) \bd s$ is said to be a \emph{Young's function}.
				
				\item Let $\psi(s) := \inf\set{t}{\varphi(t) \geq s}$. Then, the function $\Psi$ defined by $\Psi(t) := \int_0^t \psi(s) \bd s$ is said to be the \emph{complementary Young's function} of $\Phi$.
			\end{enumerate}
		\end{definition}
	
		By definition, Young's functions are convex and for a Young's function $\Phi$ it holds that the complementary Young's function $\Psi$ is also a Young's function and actually equal to the convex conjugate $\Phi^*$.
		
		The negative entropy regularization uses the function $\Phi(t) = t\log(t)$ which is not a Young's function, but the function $t\mapsto (t\log(t))_{+}$ is. Hence, we introduce a slight generalization of the notion of Young's function to be able to treat this case as well.
	
		\begin{definition}[Quasi-Young's functions]\label{def:qyfs}
            We say that $\Phi$ is a \emph{quasi-Young's function} if it is convex, lower semi-continuous and $\pos\Phi$ is a Young's function.
		\end{definition}
                
        Note that convexity of $\Phi$ shows that $\Phi$ is bounded from below.
        Moreover, any Young's function is also a quasi-Young's function.
                
		\begin{example}
			The function $\Phi(t) = t\log(t)$ is a quasi-Young's function because $\pos\Phi(t) = \pos{(t\log(t))}$ is a Young's function.
		\end{example}
	
		\begin{definition}[Luxemburg and Orlicz spaces {\cite[Definition IV.8.10]{bennett:1988}}]
			Let $\Phi$ be a Young's function, $\Omega\subset\R^n$ and $\nu\in\Mp(\Omega)$. Define the \emph{Luxemburg norm} of a measurable function $f: \Omega \to \R$ w.r.t. $\nu$ as
			\[
				\norm{f}_{L^\Phi(\Omega,\bd\nu)} := \inf \set{\gamma\geq0 }{ \intO \Phi\of{\frac{\abs{f}}{\gamma}}\dnu \leq 1}\,.
			\]
			Then the space
			\[
				\Lorl(\Omega,\,\dnu) := \Big\{f:\Omega\to\R\,\mathrm{ measurable}\,\Big|\,\norm{f}_{L^\Phi(\Omega,\,\dnu)} < \infty\Big\}
			\]
			of measurable functions on $\Omega$ with finite Luxemburg norm is called the \emph{Orlicz space of $\Phi$ w.r.t. $\nu$}.
		\end{definition}
		
		\begin{remark}[{\cite[Remark~1]{caruso:2003}}]
			\label{note:luxnorm_1}
			The bound $1$ in the definition of the Luxemburg norm can be replaced by any $a~\in~(0,\infty)$. That is, all norms defined by
			\[
				{\norm{f}}_{\Lorl(\Omega,\,\dnu), a} := \inf \set{\gamma\geq 0}{\intO \Phi\of{\frac{\abs{f}}{\gamma}}\dnu \leq a}
			\]
			are equivalent. This can be seen by combining the inequalities ${\norm{\blank}}_{\Lorl(\Omega,\,\dnu), b}~\leq~{\norm{\blank}}_{\Lorl(\Omega,\,\dnu), a}$ and $a {\norm{\blank}}_{\Lorl(\Omega,\,\dnu), a}~\leq~b{\norm{\blank}}_{\Lorl(\Omega,\,\dnu), b}$ for $0<a<b$.
		\end{remark}

		The definition of Orlicz spaces does not immediately allow for the concept of quasi-Young's functions to be incorporated. However, the following results establish the desired connection.
		
		\begin{lemma}\label{thm:lorl:eq_yf}
			Let $\Phi$ be a Young's function and $\Omega\subset \R^n$, $\nu\in\Mp(\Omega)$ with $\nu(\Omega)<\infty$ and let $t_{0}\geq 0$. Then $\tilde{\Phi}$ defined by
			$\tilde{\Phi}(t) := \pos{(\Phi(t) - \Phi(t_{0}))}$
			is a Young's function and $\Lorl(\Omega, \dnu)=\Lorl[\tilde{\Phi}](\Omega, \dnu)$.
		\end{lemma}
	    A proof of \cref{thm:lorl:eq_yf} can be found in \cref{appendix:thm:lorl:eq_yf}.
	
		\begin{corollary}\label{thm:lorl:qyf_same_space}
			Let $\Phi$ be a Young's function, $\Theta$ a quasi-Young's function with $\pos\Theta=\Phi$, $\nu$ a measure and $\Omega\subset \R^n$ with $\nu(\Omega)<\infty$. Then, $f\in\Lorl(\Omega,\,\dnu)$ if and only if
			\[
				\intO \Theta\left(\frac{\abs{f}}{\gamma}\right) \dnu \leq 1
			\]
			for some $\gamma<\infty$.
		\end{corollary}
		\begin{proof}
			Let $t_0 =\inf\{t\,|\,\Theta(\tau)\geq 0\,\forall\tau\geq t\}$ and $\tilde{\Phi}(t) := \pos{(\Phi(t) - \Phi(t_{0}))}$. Let $\gamma$ be such that
			\[
				\intO \tilde{\Phi}\left(\frac{\abs{f}}{\gamma}\right) \dnu \leq 1.
			\]
			Then
			\[
				1 \geq \intO \tilde{\Phi}\left(\frac{\abs{f}}{\gamma}\right) \dnu
			\geq \intO \Theta\left(\frac{\abs{f}}{\gamma}\right) \dnu
			\]
			together with \cref{thm:lorl:eq_yf}, shows one implication.
			
			For the other implication, by \cref{note:luxnorm_1} it suffices to show that $\intO \tilde{\Phi}\left(\nicefrac{\abs{f}}{\gamma}\right) \dnu < \infty$ whenever $\intO \Theta\left(\nicefrac{\abs{f}}{\gamma}\right) \dnu <  \infty$. However, this holds trivially, since $\Theta$ is bounded from below and $\Omega$ has finite measure w.r.t. $\nu$.
		\end{proof}
	
		\Cref{thm:lorl:eq_yf,thm:lorl:qyf_same_space} state that the definitions of $\norm{\blank}_{\Lorl(\Omega,\dnu)}$ and $\Lorl(\Omega,\dnu)$ are essentially independent of whether $\Phi$ is a Young's function or just a quasi-Young's function. To simplify notation, $\norm{\blank}_{\Lorl[\Theta](\Omega,\dnu)}$ and $\Lorl[\Theta](\Omega,\dnu)$ will therefore be used for quasi-Young's functions $\Theta$ as well. Note that for a quasi-Young's function $\Theta$ with $\pos\Theta=\Phi$, $\Lorl[\Theta](\Omega,\dnu) = \Lorl(\Omega,\dnu)$, while in general $\norm{\blank}_{\Lorl[\Theta](\Omega,\dnu)}$ and $\norm{\blank}_{\Lorl(\Omega,\dnu)}$ are equivalent but not equal.
		Moreover for complementary Young's function{s} $\Phi$ and $\Psi=\Phi^*$ which are  proper, locally integrable and satisfy a certain growth condition, called the $\Delta_2$-property near infinity, it holds that $(\Lorl{}(\Omega,\dnu))^*$ is canonically isometrically isomorphic to $(\Lorl[\Psi](\Omega,\dnu),\norm{\blank}_{\Lorl[\Psi](\Omega,\dnu)})$ (see, e.g.~\cite{diening:2011}).
		
		\begin{example}[\LlogL{} and \Lexp{}]
			Let $\Phi(t) = t \log t$ and $\tilde{\Phi}(t) = (t \log (t))_+$. The space of measurable functions $f$ with $\intO \tilde{\Phi} (\abs{f})\dnu<\infty$ is called $\LlogL(\Omega,\,\dnu)$. By the above corollary, the space of measurable functions $g$ with $\intO {\Phi}(\abs{g})\dnu<\infty$ is equal to  $\LlogL(\Omega,\,\dnu)$. The complementary Young's function $\tilde{\Psi}$ of $\tilde{\Phi}$ is given by
			\[
				\tilde{\Psi}(t) = \begin{cases}
					t,&t\leq 1\\
					\mathrm{e}^{t-1},&\text{else.}
				\end{cases}
			\]
			As $\Phi$ satisfies the $\Delta_2$-property near infinity, the dual space of $\LlogL(\Omega,\,\dnu)$ is thus given by the space of measurable functions $h$ that satisfy $\intO \tilde{\Psi}(\abs{h})\dnu<\infty$, which is called $\Lexp(\Omega,\,\dnu)$.
		\end{example}
		
		The following result states that the marginals of a transport plan with density in $\Lorl$ also have density in the respective $\Lorl$ space.
		
		\begin{lemma}\label{thm:proj_contraction}
				  Let $\nu_i\in\Mp(\Omega_i)$ be such that $\nu_i(\Omega_i)<\infty$, for $i=1,2$ and set $\nu := \nu_1 \otimes \nu_2$. Let $\pi\in\Mp(\Omega)$.

                  If $\tfrac{\dpi}{\dnu}\in\Lorl(\Omega,\dnu)$
                  for a quasi-Young's function $\Phi$, then
                  $\tfrac{\bd\ppfw
                    \pi}{\dnu_{i}}\in\Lorl(\Omega_i,\dnu_i)$
                  for $i=1,2$ with
                  \[
                  \big\lVert\tfrac{\bd\ppfw \pi}{\dnu_{i}}\big\rVert_{\Lorl(\Omega_i,d\nu_i)}
                  	\leq \max\of{1, \nu_{3-i}(\Omega_{3-i})} \big\lVert{\tfrac{\dpi}{\dnu}}\big\rVert_{\Lorl(\Omega,\dnu)}\,.
                  \]
                  \end{lemma}
        \begin{proof}
            Setting $\tfrac{\dpi}{\bd(\nu_{1}\otimes\nu_{2})}(x_{1},x_{2}) = f(x_{1},x_{2})$, one observes, that $\tfrac{\bd\ppfw[2]\pi}{\dnu_{2}} = \int_{\Omega_{1}}f(x_{1},x_{2})\dnu_{1}(x_1)$, and similarly for $\tfrac{\bd\ppfw[1]\pi}{\dnu_{1}}$. The proof given for \cite[Lemma 2.11]{clason:2021} then holds with minor modifications.
        \end{proof}
		
		\begin{remark}
			The above result immediately yields that no optimal solution $\bar\pi$ of problem \eqref{eq:reg_kantorovich} can exist for any $\gamma >0$, if $\mu_i\not\ll\lambda_{i}$ for $i\in\{1,2\}$.
		\end{remark}
		
		Next, a few facts are derived, which will be useful for the analysis of both the primal and dual regularized optimal transport problems and $\Gamma$-convergence.
		
		\begin{lemma}\label{thm:lorl_norm_lower_est}
			Let $\Phi$ be a quasi-Young's function and $f$ such that $\norm{f}_{\Lorl(\Omega,\,\dnu)} > 1$. Then,
			$\intO \Phi(\abs{f}) \dnu \geq \norm{f}_{\Lorl(\Omega,\,\dnu)}$.
		\end{lemma}
        \begin{proof}
            Noting that $\Phi(0) \leq 0$, the proof given for \cite[Lemma 2.6]{clason:2021} holds with minor adjustments.
        \end{proof}
	
		For complementary Young's's functions $\Phi$ and $\Psi$, we have the following result on conjugating scaled Young's functions.
	
		\begin{lemma}\label{thm:yf:conj_of_gammaphi}
			Let $\Phi$ be a Young's function, $\Psi=\Phi^*$ and $\gamma\neq 0$. 
			\begin{enumerate}
				\item Then for $\tilde{\Psi}(t) := \gamma \Psi(\nicefrac t\gamma)$ it holds that $\tilde{\Psi}^* = \gamma\Phi$.
				\item Let $\Theta$ be a quasi-Young's function with $\pos\Theta=\Phi$. Then for $\widetilde{\Theta^*}(t) := \gamma \Theta^*(\nicefrac t\gamma)$ it holds that $(\widetilde{\Theta^*})^* = \gamma\Theta$.
			\end{enumerate}
		\end{lemma}
        \begin{proof}
        The assertion follows directly from the definition of the convex conjugate and the Fenchel-Moreau theorem.
        \end{proof}
	
		The following \cref{thm:ws_liminf} gives some useful insights into the behavior of the objective function of problem \eqref{eq:reg_kantorovich} under perturbation of $\pi$. Recall from \cref{assmpt:general} that for Young's functions $\Phi$ we required
		$\lim_{t\to\infty} \nicefrac{\Phi(t)}{t} = \infty$.
	
		\begin{lemma}\label{thm:ws_liminf}
			Let $\nu\in\Mp(\Omega)$, $(\pi_k) \subset \P(\Omega)$ and $\pi\in\P(\Omega)$ such that $\pi_k \wsto \pi$ and $\pi_{k}\ll\nu$. Let
			$g := \ext\Phi$
			for some quasi-Young's function $\Phi$.
			Then the following statements hold.
			\begin{enumerate}
				\item
				Let $\pi\not\ll\nu$. Then
				\[
					\liminf_{k\to\infty} \intO g\of{\frac{\dpi_k}{\dnu}} \dnu = \infty\,.
				\]
				\item
				Let $\pi\ll\nu$. Then
				\[
					\liminf_{k\to\infty} \intO g\of{\frac{\dpi_k}{\dnu}} \dnu\geq  \intO g\of{\frac{\dpi}{\dnu}}\dnu\,.
				\]
			\end{enumerate}
		\end{lemma}
		\begin{proof}
			Since $g$ grows superlinearly at $\infty$, the recession function $g^{\infty}(t) =  \lim_{h\to\infty} \frac{g(s + ht) - g (s)}{h}$ (which is independent of $s$) is infinite for all $t>0$. By \cite[Theorem 5.19]{fonseca:2007}, it holds for every sequence $\pi_{k}$ which weakly* converges to $\pi$  that
			\[
				\liminf_{k\to\infty} \intO g(\tfrac{\dpi_{k}}{\dnu})\dnu \geq \intO g(\tfrac{\bd\pi}{\bd\nu})\dnu + \intO g^\infty(\tfrac{\bd\pi}{\bd\abs{\pi_{\mathrm{s}}}})\bd{\abs{\pi_s}}\,,
			\]
			where $\pi_{\mathrm{s}}$ denotes the unique measure singular to $\nu$ in the Lebesgue decomposition $\pi = \pi_{\mathrm s} + \pi_{\mathrm{ac}}$ (e.g. \cite[Theorem 1.115]{fonseca:2007}) and $\pi_{\mathrm{ac}}$ denotes the corresponding measure with $\pi_{\mathrm{ac}}\ll\nu$.
			\begin{enumerate}
				\item%
				It suffices to show $\frac{\bd\pi}{\bd\abs{\pi_{\mathrm{s}}}}(x) >0$ for every $x\in\spt \abs{\pi_{\mathrm{s}}}$.
				First note that since $\pi\in\P(\Omega)$, $\pi_{\mathrm{s}}$ is non-negative and hence $\abs{\pi_{\mathrm{s}}} = \pi_{\mathrm{s}}$. Let now $C$ be a bounded, convex closed set containing the origin in its interior. Then by \cite[Definition 1.156]{fonseca:2007}, for every $x\in\spt \pi_{\mathrm{s}}$ it holds that
				\begin{align*}
					\frac{\bd\pi}{\bd \pi_{\mathrm{s}}}(x) &= \lim_{r\searrow 0} \frac{\pi((x + rC) \cap \Omega)}{\pi_{\mathrm{s}}((x + rC) \cap \Omega)}\\
					&=  \lim_{r\searrow 0} \frac{\pi_{\mathrm{ac}}((x + rC) \cap \Omega) + \pi_{s}((x + rC) \cap \Omega)}{\pi_{\mathrm{s}}((x + rC) \cap \Omega)}\\
					&=  \lim_{r\searrow 0} \frac{\pi_{\mathrm{ac}}((x + rC) \cap \Omega)}{\pi_{\mathrm{s}}((x + rC) \cap \Omega)} + 1\\
					&\geq 1\,,
				\end{align*}
				since $\pi_{\mathrm{s}}((x + rC) \cap \Omega)>0$ for all $r$ because of $x\in\spt \pi_{\mathrm{s}}$. In fact,
				\[
					\lim_{r\searrow 0} \pi_{\mathrm{ac}}((x + rC) \cap \Omega)=0
				\]
				and  $\frac{\bd\pi}{\bd\abs{\pi_{\mathrm{s}}}}(x) = 1$ for every $x\in\spt \abs{\pi_{\mathrm{s}}}$.
				
				\item%
				The second statement follows directly, since for $\pi\ll\nu$ it holds that $\pi_{\mathrm{s}} = 0$. \qedhere
			\end{enumerate}
		\end{proof}
	
	\section{Existence of Solutions}\label{sec:existence}

		In this section, we show strong duality for the regularized mass transport \eqref{eq:reg_kantorovich} using Fenchel duality in the spaces $\P(\Omega)$ and $\CC(\Omega)$. The result will then be used to study the question of existence of solutions for both the primal and the dual problem.
		\begin{theorem}[Strong duality]\label{thm:str:dual}
			Let $\Phi$ be a quasi-Young's function and \cref{assmpt:general} hold.
			If $(\ext{\Phi})^*\of{\nicefrac{-c}{\gamma}}$ is integrable w.r.t. $\lambda$, then the predual problem to \eqref{eq:reg_kantorovich} is
			\begin{equation}\tag{P*}\label{eq:dual_reg_kantorovich}
				\sup_{\substack{\alpha_i\in\CC(\Omega_i),\\i=1,2}} \intO[1] \alpha_1 \dmu_1 + \intO[2]\alpha_2\dmu_2
				- \gamma\intO (\ext\Phi)^*\of{\frac{{\alpha_1}\oplus{\alpha_2}-c}{\gamma}} \dlambda
			\end{equation}
			and strong duality holds. Furthermore, if the supremum is finite, \eqref{eq:reg_kantorovich} possesses a minimizer.
		\end{theorem}
		\begin{proof}
			Strong duality holds by standard arguments (see e.g. \cite[Theorem 4.4.3]{borwein:2005}) and (assuming a finiteness of the supremum) the primal problem \eqref{eq:reg_kantorovich} possesses a minimizer.
                        
                        To derive the dual problem, we start from the primal problem, express the equality conditions $\int_{\Omega_{2}}\tfrac{\dpi}{\dlambda}\dlambda_{2}=\tfrac{\dmu_{1}}{\dlambda_{1}}$ and $\int_{\Omega_{1}}\tfrac{\dpi}{\dlambda}\dlambda_{1} = \tfrac{\dmu_{2}}{\dlambda_{2}}$ as suprema over continuous functions and get
                        \begin{multline*}
                          \inf_{
                            \substack{\tfrac{\dpi}{\dlambda}\in L^{\Phi}(\Omega,\dlambda),\\
                              \int_{\Omega_{2}}\tfrac{\dpi}{\dlambda}\dlambda_{2} = \tfrac{\dmu_{1}}{\dlambda_{1}},\\
                              \int_{\Omega_{1}}\tfrac{\dpi}{\dlambda}\dlambda_{1}
                              = \tfrac{\dmu_{2}}{\dlambda_{2}}}}
                          \int_{\Omega}c\tfrac{\dpi}{\dlambda} +
                          \gamma\tilde\Phi(\tfrac{\dpi}{\dlambda})\dlambda
                          \\ = \inf_{\tfrac{\dpi}{\dlambda}\in L^{\Phi}(\Omega,\dlambda)}\sup_{\substack{\alpha_i\in\CC(\Omega_i),\\i=1,2}}\int_{\Omega}c\tfrac{\dpi}{\dlambda} + \gamma\tilde\Phi(\tfrac{\dpi}{\dlambda})\dlambda + \int_{\Omega_{1}}\Big(\tfrac{\dmu_{1}}{\dlambda_{1}} - \int_{\Omega_{2}}\tfrac{\dpi}{\dlambda}\dlambda_{2}\Big)\alpha_{1}\dlambda_{1} + \int_{\Omega_{2}}\Big(\tfrac{\dmu_{2}}{\dlambda_{2}} - \int_{\Omega_{1}}\tfrac{\dpi}{\dlambda}\dlambda_{1}\Big)\alpha_{2}\dlambda_{2}\\
                          = \sup_{\substack{\alpha_i\in\CC(\Omega_i),\\i=1,2}}\inf_{\tfrac{\dpi}{\dlambda}\in L^{\Phi}(\Omega,\dlambda)}\Big( \int_{\Omega}(c - \alpha_{1}\oplus\alpha_{2}) \tfrac{\dpi}{\dlambda} + \gamma\tilde\Phi(\tfrac{\dpi}{\dlambda})\dlambda \Big) + \int_{\Omega_{1}}\alpha_{1}\dmu_{1} + \int_{\Omega_{2}}\alpha_{2}\dmu_{2}
                        \end{multline*}
                        The integrand of the first integral in \eqref{eq:dual_reg_kantorovich} is normal, so that it can be conjugated pointwise \cite[Theorem 2]{rockafellar:1968}.
			Carrying out the conjugation with the help of \cref{thm:yf:conj_of_gammaphi}, one obtains the claim.
		\end{proof}
		
		\begin{example}
			\begin{enumerate}
				\item Using $\Phi(t) = t\log t$ and $\lambda_i = \leb_i$, one obtains the result for \LlogL{} as stated in \cite[Theorem 3.1]{clason:2021}. In this case it holds that $(\ext\Phi)^*(r) = \exp(r) = \Phi^{*}(r)$.
				\item Using $\Phi(t) = \frac{1}{2} t^2$ and $\lambda_i = \leb_i$, one obtains the result for $L^2$ as stated in \cite{lorenz:2019}.
                                  In this case it holds that $(\ext\Phi)^*(r) = \max(0,r)^{2}$.
			\end{enumerate}
                        One can show that in general $(\ext\Phi)^{*}(r) = \Phi^{*}(r)$ for $r\geq\inf_{\tau>0}\partial\Phi(\tau)$ and equal to $-\Phi(0)$ otherwise.
		\end{example}
	
		\begin{remark}
			\Cref{thm:str:dual} does \emph{not} claim that the supremum is attained, i.e. that the predual problem \eqref{eq:dual_reg_kantorovich} admits a solution. Moreover, the solutions of \eqref{eq:dual_reg_kantorovich} cannot be unique since one can add and subtract constants to $\alpha_1$ and $\alpha_2$,	respectively, without changing the functional value. If however $(\ext\Phi)^*$ is strictly convex, then the functional in \eqref{eq:dual_reg_kantorovich} is strictly concave up to such a constant and therefore any solution is uniquely determined by this constant. This is the case, e.g., for functions $\Phi$ with superlinear growth at $\infty$.
		\end{remark}
	
	\subsection{Existence result for the primal problem}
		
		The duality result can now be used to address the question of existence of a solution to \eqref{eq:reg_kantorovich}.
		
		\begin{theorem}[Existence of solutions of \eqref{eq:reg_kantorovich}]\label{thm:primal_existence}
            Let \cref{assmpt:general} hold.
            If $(\ext{\Phi})^*\of{\nicefrac{-c}{\gamma}}$ is integrable w.r.t. $\lambda$ and
			\begin{equation}\label{cond:prod_of_fcts}
				\tfrac{\dmu_i}{\dlambda_{i}} \in\Lorl(\Omega_i,\dlambda_{i}),\,i=1,2\,\Rightarrow\,\tfrac{\bd(\mu_1\otimes \mu_2)}{\dlambda} \in\Lorl(\Omega,\dlambda)
			\end{equation}
                        we have that problem
                        \eqref{eq:reg_kantorovich} admits a minimizer
                        $\bar{\pi}$ if and only if
                        $\tfrac{\dmu_i}{\dlambda_{i}}
                        \in\Lorl(\Omega_i,\dlambda_{i})$
                        for $i=1,2$. In this case,
                        $\bar{\pi} \in\Lorl(\Omega,\dlambda)$ and the minimizer
                        is unique, if $\Phi$ is strictly
                        convex.
		\end{theorem}
		\begin{proof}
			The proof given in \cite[Theorem 3.3]{clason:2021} for $\Phi(t) = t \log t$ holds for arbitrary $\Phi$ and arbitrary product measures $\lambda = \lambda_{1}\otimes\lambda_{2}$,
			 since it only relies on \cref{thm:proj_contraction} and condition \eqref{cond:prod_of_fcts}.
		\end{proof}
	
		\begin{example}
			Since $\tfrac{\bd(\mu_{1}\otimes\mu_{2})}{\dlambda} = \tfrac{\dmu_{1}}{\dlambda_{1}}\otimes\tfrac{\dmu_{2}}{\dlambda_{2}}$, condition \eqref{cond:prod_of_fcts} is satisfied e.g. when $\Phi$ satisfies either $\Phi(xy) \leq C \Phi(x)\Phi(y)$ for some $C>0$ or
			$\Phi(xy) \leq C_1 x\Phi(y) + C_2 \Phi(x)y$
			for some $C_1,C_2 \geq 0$.
			For $\Phi(t) = \nicefrac{t^p}{p}$, $p>1$, both conditions hold trivially. For $\Phi(t) = t\log t$ the second condition holds, since $\log(xy) = \log(x) + \log(y)$.
		\end{example}
	
	\subsection{Existence result for the predual problem with \texorpdfstring{$L^{p}(\Omega,\dlambda)$}{Lp(\unichar{"03A9}, d\unichar{"03BB})} regularization}
	
		The question of existence of solutions to the predual problem \eqref{eq:dual_reg_kantorovich} proves to be difficult for general Young's functions. 
		There are results that show existence for the predual problem in the entropic case~\cite{clason:2021} and in the quadratic case~\cite{lorenz:2019} (both considering the penalty w.r.t. the Lebesgue measure), but their proofs are quite different in nature. 
                In the case where $\lambda = \mu_{1}\otimes\mu_{2}$ (i.e. the product of the marginals) and $\Phi(t) = t\log(t)$, one can show dual existence in a different way, see~\cite{genevay:2019entropy} for a proof based on the convergence of the Sinkhorn method and~\cite{vialard:2019} for a sketch of a proof that uses regularity of the cost function.
                Our methods do not allow to show existence of solutions for the dual problem in the general case considered up to now. Hence, we consider a special case in the following:
                We specialize to the case of $L^{p}(\Omega,\dlambda)$-regularization for a measure $\lambda=\lambda_{1}\otimes\lambda_{2}$ with $\mu_{i}\ll\lambda_{i}$, $i=1,2$.                 
		That is, we use the Young's functions $\Phi(t)= \nicefrac{t^p}{p}$ for $p>1$, and thus, ${\Phi}^*(t) = \nicefrac{t^q}{q}$, with $\nicefrac 1 p + \nicefrac 1 q = 1$, and the predual is actually also the dual.
                Moreover, $(\ext\Phi)^*(s) = \tfrac1q(\pos{s})^{q}$.
		To keep notation clean, $(\pos t)^p$ will abbreviated as $\pos{t}^p$ with slight abuse of notation and similarly for $(\negprt t)^p$.
		
		\begin{assumption}\label{asspt:dual}
			Let \cref{assmpt:general} hold. In addition, let $\Phi(t) = \nicefrac{t^p}{p}$ for $p>1$ and let the cost function $c$ be continuous and fulfill $c \geq c^{\dagger}$ for some constant $c^\dagger > -\infty$. Furthermore, let the marginals $\mu_i$ with $\tfrac{\bd\mu_i}{\bd\lambda_i}\in L^p(\Omega_i,\dlambda_i)$ satisfy
			$\tfrac{\bd\mu_i}{\bd\lambda_i}\geq \delta > 0$
			$\lambda_i$-%
			a.e.
			for $i=1,2$.
		\end{assumption}
                
        Note that the latter condition, i.e that the densities should be bounded away from zero, can be guaranteed by proper choice of $\lambda_{i}$, e.g. $\lambda_{i} = \mu_{i}$ gives $\tfrac{\dmu_{i}}{\dlambda_{i}} = 1$.
        It can not be expected for problem \eqref{eq:dual_reg_kantorovich} to have continuous optimizers ${\alpha_1}$, ${\alpha_2}$, as the following example demonstrates.

        \begin{example}\label{example:no-cont-duals}
            For $i=1,2$ let $\Omega_i = [-1,1]$, $\lambda_i = \leb|_{\Omega_i} + \delta_0$, where $\delta_0$ denotes the Dirac measure at zero, and $\mu_i = \delta_0$. Moreover, let $c \equiv 0$.
            Clearly the minimum in problem \eqref{eq:reg_kantorovich} is $\frac 1p$, which is attained for $\pi = \delta_0\times\delta_0$. By \cref{thm:str:dual}, the optimal value of problem \eqref{eq:dual_reg_kantorovich} is $\frac 1p$, as well. Going on, it holds
            \begin{align*}
                & \sup_{\alpha_i\in\CC(\Omega_i)} \intO[1]\alpha_1\dmu_1 + \intO[2]\alpha_2\dmu_2 -\frac 1q \intO (\alpha_1\oplus\alpha_2)_+^q\dlambda \\
                =& \sup_{\alpha_i\in\CC(\Omega_i)} \alpha_1(0) + \alpha_2(0)  - \frac 1q (\alpha_1(0) + \alpha_2(0))_+^q - \frac 1q\intO (\alpha_1\oplus\alpha_2)_+^q\dleb
            \end{align*}
            If $\alpha_1\oplus\alpha_2\leq 0$, the supremum ($\frac 1p$) can clearly not be attained. Hence, assume $(\alpha_1\oplus\alpha_2)(x)>0$ for some $x\in\Omega$, which implies $\intO (\alpha_1\oplus\alpha_2)_+^q\dleb>0$, as $\alpha_i\in\CC(\Omega_i)$. Thus,
            \begin{align*}
                & \alpha_1(0) + \alpha_2(0)  - \frac 1q (\alpha_1(0) + \alpha_2(0))_+^q - \frac 1q\intO (\alpha_1\oplus\alpha_2)_+^q\dleb\\
                < & (\alpha_1\oplus\alpha_2)(0)  - \frac 1q ((\alpha_1\oplus\alpha_2)(0))_+^q
                \leq 1 - \frac 1q = \frac 1p\,,
            \end{align*}
            where equality holds for $(\alpha_1\oplus\alpha_2)(0)=1$. Due to the strict inequality, the supremum in problem \eqref{eq:dual_reg_kantorovich} can \emph{not} be attained for $\alpha_i\in\CC(\Omega_i)$, $i=1,2$.
        \end{example}

        However, with the help of the following result, we can define a variant of the predual problem, for which existence of minimizers can be shown.

        \begin{lemma}\label{thm:dual:lq:pos}
            Let \cref{assmpt:general} hold and
            let $\alpha_i:\Omega_i \to \R\cup\{\pm\infty\}$, $i=1,2$ be such that $\pos{({\alpha_1}\oplus{\alpha_2} - c)} \in L^q(\Omega,\dlambda)$. Then $\pos{(\alpha_i)}\in L^q(\Omega_i, \dlambda_i)$, $i=1,2$.
        \end{lemma}
        \begin{proof}
            First note, that
			\begin{align*}
				\pos{({\alpha_1}\oplus{\alpha_2} - c + c)}^q &\leq {(\pos{({\alpha_1}\oplus{\alpha_2} - c )}+ \pos c)}^q
			    = 2^q \of{\frac{\pos{({\alpha_1}\oplus{\alpha_2} - c )}} 2 + \frac{\pos c} 2}^{\!q}\\
				&\leq 2^{q-1}\of{ \pos{({\alpha_1}\oplus {\alpha_2} -c)}^q + \pos c^q}\,,
			\end{align*}
            so that $\pos{(\alpha_1\oplus\alpha_2)}\in L^q(\Omega, \dlambda)$. This implies $\pos{(\alpha_i)}<\infty$ $\lambda_i$-a.e. for $i=1,2$.

            We now consider $\pos{(\alpha_i)}$, $i=1,2$ separately and start with $i=1$.
            Let $M\subset\Omega_2$ be such that $\alpha_2 \geq 0$ $\lambda_2$-a.e. on $M$. If $\lambda_2(M)=0$ the assertion holds trivially, so we assume $\lambda_2(M)>0$.
            Then,
            \begin{align*}
                \norm{\pos{(\alpha_1)}}_q
                    &= \sup\left\{\intO[1]\pos{(\alpha_1)}g_1\dlambda_1\,\middle|\,g_1\in L^p(\Omega_1, \dlambda_1), \norm{g_1}_p = 1\right\}\\
                    & 
                    \begin{multlined}
                        = \sup \left\{
                        \lambda_2(M)^{\frac 1p - 1} \int_M\intO[1]\pos{(\alpha_1)}(g_1\otimes g_2)\dlambda_1\dlambda_2
                        \,\middle|
                        \vphantom{\lambda_2(M)^{-\tfrac 1p}}\right.\\
                        \hphantom{\sup \left\{ \int_M\intO[1]\pos{(\alpha_1)} \right.}\left.
                        g_1\in L^p(\Omega_1, \dlambda_1), \norm{g_1}_p = 1, g_2 = \1_M\lambda_2(M)^{-\tfrac 1p} \right\}
                    \end{multlined}\\
                    &\leq \lambda_2(M)^{\frac 1p - 1}\cdot \sup \left\{  \int_M\intO[1]\pos{(\alpha_1\oplus\alpha_2)}(g_1\otimes g_2)\dlambda_1\dlambda_2 \,\middle|\, g_i\in L^p(\Omega_i, \dlambda_i), \norm{g_i}_p = 1, i=1,2 \right\}\\ 
                    &\leq \lambda_2(M)^{\frac 1p - 1}\cdot \sup \left\{ \intO\pos{(\alpha_1\oplus\alpha_2)}f\dlambda\,\middle|\, f\in L^p(\Omega, \dlambda), \norm{f}_p = 1 \right\}\\ 
                    &= \lambda_2(M)^{\frac 1p - 1}\cdot \norm{\pos{(\alpha_1\oplus\alpha_2)}}_q < \infty\,,
            \end{align*}
            where the first inequality is justified because thanks to $\pos{(\alpha_1)}\geq 0$ we can take the supremum over $g_1\geq 0$.
            With the same argumentation, we can prove $\pos{(\alpha_2)}\in L^q(\Omega_2, \dlambda_2)$.
        \end{proof}

        Hence, the objective function of problem \eqref{eq:dual_reg_kantorovich} is also well defined for functions ${\alpha_i}\in L^1(\Omega_i,\dlambda_i)$, $i=1,2$, with $\pos{({\alpha_1}\oplus{\alpha_2} - c)} \in L^q(\Omega,\dlambda)$. We now considered the following variant of the predual problem:
		\begin{equation}\tag{P$^\dagger$}\label{eq:dual_aux}
			\begin{multlined}
				\min\left\{%
					\Lambda({\alpha_1},{\alpha_2}) := \frac{1}{q} \norm{\pos{({\alpha_1}\oplus {\alpha_2} -c)}}^q_{L^q(\Omega,\,\dlambda)}
					\vspace*{-1em}%
					- \gamma^{q-1}\intO[1] {\alpha_1} \dmu_1 - \gamma^{q-1}\intO[2]{\alpha_2}\dmu_2
				\right.\\
				\left.\vphantom{\intO[2]{\alpha_2}}\middle|\,%
					\alpha_i\in L^1(\Omega_i,\dlambda_i),\,i=1,2,\,\tfrac1\gamma\pos{({\alpha_1}\oplus{\alpha_2} - c)} \in L^q(\Omega,\dlambda)%
				\right\}\,.%
			\end{multlined}
		\end{equation}
		The strategy in this section is as follows.
		\begin{enumerate}
			\item First, show that problem \eqref{eq:dual_aux} admits a solution $({\bar{\alpha}_1},{\bar{\alpha}_2}) \in L^1(\Omega_1,\dlambda_1) \times L^1(\Omega_2,\dlambda_2)$.
			\item Then, prove that ${\bar{\alpha}_1}$ and ${\bar{\alpha}_2}$ possess higher regularity, namely that they are functions in $L^q(\Omega_i,\dlambda_i)$.
		\end{enumerate}
		The objective function needs to be extended to allow to deal with weak-$*$ converging sequences. To that end, define
		\[
			G:L^q(\Omega,\dlambda)\ni w \mapsto \intO \of{\frac{1}{q} \pos{w}^q - w \mu }\dlambda \in\R \,,
		\]
		where $\mu := \gamma^{q-1}\tfrac{\bd(\mu_1\otimes\mu_2)}{\bd\lambda}$. Note that in the case $\lambda_i = \mu_i$, the variable $\mu$ is given by $1 \cdot \gamma^{q-1}$. Then, thanks to the normalization of ${\mu_1}$ and ${\mu_2}$,
		\[
			\Lambda({\alpha_1},{\alpha_2}) = G({\alpha_1}\oplus {\alpha_2} - c) - \intO c \mu \dlambda\quad \forall{\alpha_1},{\alpha_2} \in L^q\,.
		\]
		Of course, $G$ is also well defined as a functional on the feasible set of problem \eqref{eq:dual_aux} and this functional will be denoted by the same symbol to ease notation. In order to extend $G$ to the space of Radon measures, consider for a given measure $w \in\M(\Omega)$ the Hahn-Jordan decomposition $w=\pos w - \negprt w$ and assume $\tfrac{\bd\pos w}{\dlambda} \in L^q(\Omega,\dlambda)$. Then, $\intO \tfrac{\bd(w_+\mu)}{\dlambda}\dlambda$ is finite for $\mu$ as defined above and $\mu_i$ as in \cref{asspt:dual}. Regarding the negative part, we set $\intO \mu\bd w_-:=\infty$, whenever this expression is not properly defined, as $w_-$ and $\mu$ are both positive. Combining this, we always have $-\intO\mu\bd w\in\R\cup\{\infty\}$ and define
		\[
			G(w) := \intO \frac{1}{q} \frac{\bd\pos{w}^q}{\dlambda} \dlambda - \intO \mu \dw\,.
		\]
		
		\begin{remark}
			If $w\ll \lambda$, then $\tfrac{\bd\pos w}{\dlambda} \in L^1(\Omega,\dlambda)$ and $\tfrac{\bd\pos w}{\dlambda}(x) = \max\{0, \tfrac{\dw}{\dlambda}(x)\}$ $\lambda$-a.e. in $\Omega$. Hence, both functionals denoted by $G$ coincide on $L^q(\Omega,\dlambda)$, which justifies this notation.
		\end{remark}
		
		The following auxiliary results are generalizations of the corresponding results in \cite{lorenz:2019}.
		The first lemma covers the coercivity of $G$ in $L^1(\Omega,\dlambda)$. To keep notation simple, from now on we will abbreviate $\norm{\blank}_{L^p(\Omega, \dlambda)}$ by $\norm{\blank}_p$ and similarly for $\norm{\blank}_{L^p(\Omega_i, \dlambda_i)}$, $i = 1,2$, where the underlying space will be clear from the context.
		
		\begin{lemma}\label{thm:dual:boundedness}
			Let \cref{asspt:dual} hold and suppose that a sequence $(w_n)\subset L^q (\Omega,\dlambda)$ fulfills
			\[
				G(w_n) \leq C < \infty\quad \forall n \in\N  
			\]
			for some $C>0$.
			Then, the sequences $\pos{(w_n)}$ and $\negprt{(w_n)}$ are bounded in $L^q(\Omega,\dlambda)$ and $L^1 (\Omega,\dlambda)$, respectively.
		\end{lemma}
		\begin{proof}
			This proof follows the outline of the proof in \cite[Lemma 2.5]{lorenz:2019}.
			
			We rewrite $G$ as $G(w) = \intO \frac1q \pos{w}^q - \pos w \mu \dlambda + \intO \negprt{w} \mu \dlambda$. The positivity of $\mu$ implies
			\[
				\frac 1q\norm{\pos{(w_n)}}_q^q = G(w_n) + \intO \pos{(w_n)} \mu \dlambda - \intO \negprt{(w_n)} \mu \dlambda \leq C + \norm{\mu}_p\norm{\pos{(w_n)}}_q\,,
			\]
			which gives the first assertion.
			The second one can be seen by making use of $\mu\geq \gamma^{q-1}\delta^2$ with $\delta$ from \cref{asspt:dual}, which yields the estimate
			\begin{align*}
				C \geq G(w_n) & = \frac{1}{q}\intO \pos{(w_n)}^q \dlambda - \intO \pos{(w_n)} \mu\dlambda + \intO \negprt{(w_n)} \mu\dlambda\\
				&\geq \frac{1}{q} \norm{\pos{(w_n)}}_q^q - \norm{\mu}_p\norm{\pos{(w_n)}}_q + \gamma^{q-1}\delta^2 \norm{\negprt{(w_n)}}_1\\
				&\geq  - \norm{\mu}_p\norm{\pos{(w_n)}}_q + \gamma^{q-1}\delta^2 \norm{\negprt{(w_n)}}_1\,.
			\end{align*}
			
			Since $\norm{\pos{(w_n)}}_q$ is already known to be bounded, the second assertion holds.
		\end{proof}
		
		The next lemma provides a lower semi-continuity result for $G$ w.r.t. weak-$*$ convergence in $\M(\Omega)$. Note that the extension of $G$ as introduced above is needed, here.
		
		\begin{lemma}\label{thm:dual:wsconv}
			Let \cref{asspt:dual} hold and a sequence $(w_n)\subset L^q (\Omega,\dlambda)$ be given such that $w_n\lambda \wsto \bar{w}$ in $\M(\Omega)$ and $G(w_n) \leq C < \infty$ for all $n\in\N$. Then it holds that $\pos{\bar w}\ll \lambda$ with $\tfrac{\pos{\bd\bar w}}{\dlambda}\in L^q (\Omega,\dlambda)$ and
			\[
				G(\bar{w}) \leq \liminf_{n\to\infty} G(w_n).
			\]
		\end{lemma}
		\begin{proof}
			The proof given in \cite[Lemma 2.6]{lorenz:2019} only uses \cref{thm:dual:boundedness} and fundamental properties of $L^2(\Omega,\bd\leb)$, which also hold for $L^q(\Omega, \dlambda)$, $q> 1$, and can thus be readily extended.
		\end{proof}
		
		Now
		all prerequisites for proving the existence result for problem \eqref{eq:dual_aux} are gathered.
		
		\begin{proposition}\label{thm:dual-l1-solultion}
			Let \cref{asspt:dual} hold. Then, problem \eqref{eq:dual_aux} admits a solution $({\bar{\alpha}_1},{\bar{\alpha}_2})\in L^1 (\Omega_1,\dlambda_1)\times L^1 (\Omega_2,\dlambda_2)$.
		\end{proposition} 
		\begin{proof}
			In \cite[Proposition 2.9]{lorenz:2019} the statement is proven for $p=2$ via the classical direct method of the calculus of variations using only
			\cite[Lemmas 2.7 \& 2.8]{lorenz:2019} and \cref{thm:dual:boundedness,thm:dual:wsconv}, where
			\cite[Lemmas 2.7 \& 2.8]{lorenz:2019} are rather technical results holding independently of the choice of $p>1$ and $\lambda_1$, $\lambda_2$.
			Hence, the proof also holds for $p>1$.
		\end{proof}
		
		Next, it is shown that $\alpha_i$, $i=1,2$ are indeed functions in $L^q(\Omega_i,\dlambda_i)$.
		
		\begin{theorem}\label{thm:dual:lq}
			Let \cref{asspt:dual} hold and let $p\geq 2$. Then every optimal solution $({\bar{\alpha}_1}, {\bar{\alpha}_2})$ from \cref{thm:dual-l1-solultion} satisfies ${\bar{\alpha}_i}\in L^q(\Omega_i,\dlambda_i)$, $i=1,2$. Moreover, the negative parts of $\bar\alpha_i$ are bounded and the function $\tfrac1{\gamma^{q-1}}{\pos{(\bar\alpha_1\oplus\bar\alpha_2 -c)}^{q-1}}$ has the marginals $\tfrac{\dmu_1}{\dlambda_1}$ and $\tfrac{\dmu_2}{\dlambda_2}$.
		\end{theorem}
		\begin{proof}
			We only consider the the negative parts, as for the positive parts we already have $\pos{(\alpha_i)}\in L^q(\Omega_i,\dlambda_i)$ by \cref{thm:dual:lq:pos}.
			
            First note, that for functions $f_i: \Omega_i \to\R$, $i=1,2$ and $g:\Omega_1\to\R$ it holds
            \[
                \pos{((f_1 + g)\oplus f_2)} = \max\Big(0,(f_1\oplus f_2) + g\oplus 0\Big) \leq \pos{(f_1\oplus f_2)} + \pos g\oplus 0\,,
            \]
            where $0$ is to be understood as the constant mapping $x_2 \mapsto 0$.
			Let now $\varphi\in\test(\Omega_1)$ and fix some $0 < t\leq 1$. Then, thanks to
			\begin{equation}\label{eq:dual_sol:lp:aux}
				0\leq \pos{((\bar{\alpha}_1 + t\varphi)\oplus \bar{\alpha}_2 - c)} \leq \pos{(\bar{\alpha}_1\oplus \bar{\alpha}_2 - c)} + t\pos\varphi \oplus 0\,,
			\end{equation}
			\cref{thm:dual-l1-solultion} implies that $\pos{((\bar{\alpha}_1 + t\varphi)\oplus \bar{\alpha}_2 - c)}\in L^q(\Omega,\dlambda)$, so that $(\bar{\alpha}_1 + t\varphi,\bar{\alpha}_2)$ is feasible for problem \eqref{eq:dual_aux}. Therefore, the optimality of $(\bar{\alpha}_1,\bar{\alpha}_2)$ for problem \eqref{eq:dual_aux} yields
			\[
				\frac{1}{q}\intO\frac 1t \of{\pos{((\bar{\alpha}_1 + t\varphi)\oplus \bar{\alpha}_2 - c)}^q - \pos{(\bar{\alpha}_1 \oplus \bar{\alpha}_2 - c)}^q}\dlambda - \gamma^{q-1}\intO[1] \tfrac{\bd\mu_1}{\dlambda_1} \varphi\dlambda_1 \geq 0
			\]
			for all $0<t\leq 1$.
			Owing to the continuous differentiability of $\R\ni r\mapsto \pos r^q\in\R$, the first integrand converges to $q\pos{(\bar{\alpha}_1\oplus \bar{\alpha}_2 - c)}^{q-1} \varphi$ $\lambda$-a.e. in $\Omega$.

			Moreover, for $x\geq 0$, the mapping $[0,1]\ni t \mapsto (x + t\pos\varphi)^q\in\R $ is Lipschitz continuous with Lipschitz constant $q \pos\varphi \of{x + \pos\varphi}^{q-1}$. Together with \eqref{eq:dual_sol:lp:aux}, this gives
			\begin{equation}\label{ineq:lq}
			\begin{split}
				\frac 1t \of{\pos{((\bar{\alpha}_1 + t\varphi)\oplus \bar{\alpha}_2 - c)}^q - \pos{(\bar{\alpha}_1 \oplus \bar{\alpha}_2 - c)}^q}
				&\leq\frac 1t \of{\of{\pos{(\bar{\alpha}_1\oplus \bar{\alpha}_2 - c)} + t\pos\varphi\oplus 0}^q - \pos{(\bar{\alpha}_1 \oplus \bar{\alpha}_2 - c)}^q}\\
				&\leq q \pos\varphi \of{\pos{(\bar{\alpha}_1 \oplus \bar{\alpha}_2 - c)} + \pos\varphi\oplus 0}^{q-1}\\
				&\leq q \pos\varphi \of{2 \max\left\{\pos{(\bar{\alpha}_1 \oplus \bar{\alpha}_2 - c)}\,,\,\pos\varphi\oplus 0\right\}}^{q-1}\\
				&\leq q \pos\varphi 2^{q-1} \of{\pos{(\bar{\alpha}_1 \oplus \bar{\alpha}_2 - c)}^{q-1} + (\pos{\varphi}\oplus 0)^{q-1}}\,,
			\end{split}
			\end{equation}
			since $(x + y)^r \leq \of{2\max\{x\,,\,y\}}^r$ for all $r>0$, $x,y\geq 0$. As $g\in L^{q}(\Omega,\dlambda)$ with $\lambda(\Omega)<\infty$ implies $\intO g^{q-1}\dlambda<\infty$ for all $q>1$ (see e.g. \cite[Proposition 6.12]{folland:1999}), the right-hand side is integrable.
			
			Hence, due to Lebesgue's dominated convergence theorem, passing	to the limit $t \searrow 0 $ is allowed and yields
			\[
				\intO[1]\of{\intO[2] \pos{(\bar{\alpha}_1\oplus \bar{\alpha}_2 - c)}^{q-1}\dlambda_2-  \gamma^{q-1}\tfrac{\bd\mu_1}{\dlambda_1}}\varphi\dlambda_1 \geq 0\,.
			\]
			Since $\varphi\in\test(\Omega_1)$ was arbitrary, the fundamental lemma of the calculus of variations gives
			\begin{equation}\label{eq:lq}
				\intO[2]\pos{(\bar{\alpha}_1 \oplus \bar{\alpha}_2 - c)}^{q-1}\dlambda_2 =  \gamma^{q-1}\tfrac{\bd\mu_1}{\dlambda_1}
			\end{equation}
			$\lambda_1$-a.e. in $\Omega_1$. Next, define a sequence of functions $(f_n)$ by
			\[
				f_n(x_2) := \pos{(-n + \bar{\alpha}_2(x_2) - c^\dagger)}\quad \forall n\in\N\,,
			\]
			where $c^\dagger$ is the lower bound for $c$ from \cref{asspt:dual}.
			It holds that $f_n \in L^{q-1}(\Omega_2,\dlambda_2)$, which can be seen as follows: Since $\bar{\alpha}_2\in L^1(\Omega_2,\dlambda_2)$, clearly $(\bar{\alpha}_2)^{q-1}\in L^{\nicefrac{1}{(q-1)}}(\Omega_2,\dlambda_2)$. Furthermore, since $\Omega_{2}$ is compact and $q\leq 2$, it holds $L^{\nicefrac{1}{(q-1)}}(\Omega_2, \dlambda_2) \hookrightarrow L^1(\Omega_2, \dlambda_2)$. Consequently, $(\bar{\alpha}_2)^{q-1}\in L^1(\Omega_2, \dlambda_2)$ and $f_{n}^{q-1}$ is also integrable for every $n$.

			The functions $f_n$ satisfy $f_n\geq 0$ and $f_n \searrow 0$ $\lambda_2$-a.e. in $\Omega_2$ for $n\to\infty$, so that the monotone convergence theorem gives
			\[
				\intO[2] f_n^{q-1} \dlambda_2 \xrightarrow[n\to \infty]{}0\,.
			\]
			Thus, there exists $N\in\N$ such that
			\[
				\intO[2] \pos{(-N + \bar{\alpha}_2 - c^\dagger)}^{q-1}\dlambda_2 < \gamma^{q-1} \delta\,,
			\]
			with the threshold $\delta>0$ from~\cref{asspt:dual}.
			Now assume that $\bar{\alpha}_1 \leq -N$ $\lambda_1$-a.e. on a set $E\subset\Omega_1$ with $\lambda_1(E)>0$. Then
			\[
				\intO[2] \pos{(\bar{\alpha}_1\oplus \bar{\alpha}_2 -c^\dagger)}^{q-1}\dlambda_2 \leq \intO[2] \pos{(-N+\bar{\alpha}_2 -c^\dagger)}^{q-1} \dlambda_2 <  \gamma^{q-1} \delta \leq \gamma^{q-1}\tfrac{\bd\mu_1}{\dlambda_1}
			\]
			$\lambda_1$-a.e. in $E$,
			which is a contradiction to \eqref{eq:lq}. Therefore $\bar{\alpha}_1 > -N$ $\lambda_1$-a.e. in $\Omega_1$, which even implies that $\negprt{(\bar{\alpha}_1)}\in L^{\infty}(\Omega_1,\dlambda_1)$. Concerning $\negprt{(\bar{\alpha}_2)}$, one may argue exactly the same way to conclude that $\negprt{(\bar{\alpha}_2)}\in L^{\infty}(\Omega_1,\dlambda_1)$, too.
		\end{proof}
		
		\begin{remark}
                  While it seems clear that proving a generalization of \cref{thm:dual:lq} to general Young's functions or even quasi-Young's functions $\Phi$ is likely to be complicated or even impossible without making strict assumptions on $\Phi$, 
		not even the existence result for optimizers in $L^1$ can be generalized directly. The problem occurs in \cref{thm:dual:boundedness},
		which could not be extended to the case of Young's functions or quasi-Young's functions $\Phi$ in this work.
		That additional assumptions on $\Phi$ might be necessary for \cref{thm:dual:boundedness} to hold can be seen as follows.
		
		In the general case, the function $G$ would be defined as
		\[
			G:L^q(\Omega,\dlambda)\ni w \mapsto \intO \of{\gamma\tilde{\Phi}^*\of{\tfrac{w}{\gamma}} - w \mu} \dlambda \in\R \,,
		\]
		where $\mu := \tfrac{\bd({\mu_1}\otimes {\mu_2})}{\dlambda}$ and in the proof of \cref{thm:dual:boundedness} an inequality of the form
		\[
			C \norm{w_n}_{\Lorl[\tilde{\Phi}^*](\Omega,\dlambda)} \leq \gamma \intO \tilde{\Phi}^*\of{\tfrac{w_n}{\gamma}}\dlambda\quad \forall n\in\N
		\]
		would be necessary. For this to hold, it would suffice to know
		\[
			\Phi\of{\norm{f}_{\Lorl(\Omega,\dlambda)}} \leq C \intO \Phi (\abs f) \dlambda
		\]
		for (quasi-)Young's functions $\Phi$, but this is not true in general as the Young's function
		$\Phi(t)=\max(t^2,t^3)$ and $\Omega=(0,1)$ shows.
                Indeed for $f=a\1_{(0,b)}$ for some $a,b\in (0,1)$ one readily computes that  $\norm{f}_{\Lorl(\Omega,\,\dleb)}= a b^{1/3}$ and the above mentioned inequality would be
                \[
                a^2 b^{2/3} \leq C a^2 b\,,
                \]
                which is not possible for any constant $C$ independent of $b$.\footnote{We thank the user \texttt{harfe} from mathoverflow who provided this counterexample to our question \url{https://mathoverflow.net/q/333925}.}
              This counterexample 
indicates that both the growth of $\Phi$ at infinity and at zero are important properties for this problem.
              \end{remark}		
	
	\section{\texorpdfstring{$\Gamma$}{\unichar{"0393}}-Convergence}\label{sec:gamma_conv}

        We return to the general setting, i.e. we only assume that \cref{assmpt:general} holds, and consider results on $\Gamma$-convergence of problems related to problem \eqref{eq:reg_kantorovich}.
		Recall from, e.g., \cite{braides:2002}, that a sequence $(F_n)$ of functionals $F_n:X\to \R\cup\{\infty\}$ on a metric space $X$ is said to $\Gamma$-converge to a functional $F:X\to\R\cup\{\infty\}$, written $F = \Gammalim_{n\to\infty} F_n$, if
		\begin{enumerate}[(i)]
			\item for every sequence $\{x_n\}\subset X$ with $x_n\to x$,
			\begin{equation*}
			F(x) \leq \liminf_{n\to\infty} F_n(x_n),
			\end{equation*}
			\item for every $x\in X$, there is a sequence $\{x_n\}\subset X$ with $x_n\to x$ and
			\begin{equation*}
			F(x) \geq \limsup_{n\to\infty} F_n(x_n).
			\end{equation*}
		\end{enumerate}
		It is a straightforward consequence of this definition that if $F_n$ $\Gamma$-converges to $F$ and $x_n$ is a minimizer of $F_n$ for every $n\in \N$, then every cluster point of the sequence $(x_n)$ is a minimizer to $F$. Furthermore, $\Gamma$-convergence is stable under perturbations by continuous functionals.
	
	\subsection{Continuous case}
		When considering arbitrary measures as marginals, their densities w.r.t. $\lambda_i$ may not be in $\Lorl(\Omega,\dlambda_i)$ and by \cref{thm:primal_existence}, problem \eqref{eq:reg_kantorovich} will not admit a solution in that case.
		One may therefore consider
		smoothed marginals $\mu^\delta_i$ with $\tfrac{\bd\mu_i^\delta}{\dlambda_i}\in\Lorl(\Omega_i,\dlambda_i)$, $i=1,2$, converging to $\mu_1$ and $\mu_2$, respectively and show that the  regularized problem with these marginals $\Gamma$-converges to the
		unregularized problem with the original marginals.
		
		Let $\varphi$ be a smooth, compactly supported, non-negative kernel on $\R^n$ with unit integral (w.r.t. the Lebesgue measure) and set
		\begin{equation*}
			\varphi_r(x)=\tfrac1{r^{n}}\varphi(\tfrac xr)\,,\qquad
			G_r:=\varphi_r\otimes \varphi_r\,.
		\end{equation*}

		Since the marginals and transport plans will be smoothed by convolutions, the domains $\Omega_1$ and $\Omega_2$ will be extended slightly to take care of boundary effects. 
		Hence, let $\tilde{\Omega}_1,\,\tilde{\Omega}_2$ be compact supersets of $\Omega_1,\,\Omega_2$, respectively, such that
		\[
			\Omega_i + \spt \varphi \subseteq \tilde{\Omega}_i,\quad i=1,2\,,
		\]
		which is large enough to contain the supports of the smoothed marginals $\mu_i^r$, $i=1,2$ for $r\leq 1$ (and the width of the convolution kernels will be assumed to be small enough for this in the following).
		For a function (or measure) $f$ on $\Omega_1$ denote by $\tilde f$ the extension of $f$ onto $\tilde{\Omega}_1$ by zero (and analogously for functions and measures on $\Omega_2$ and $\Omega_1\times\Omega_2$).
		Let $\hat\lambda_i$ be the extension of $\lambda_i$ onto $\tilde\Omega_i$ by the Lebesgue measure and $\hat\lambda = \hat\lambda_1 \otimes\hat\lambda_2$.
		Let $\hat c$ be a continuous extension of $c$ onto $\tilde{\Omega}_1\times\tilde{\Omega}_2$ and let
		\[
		F_\gamma[\pi]
			=\int_{\tilde{\Omega}_1\times\tilde{\Omega}_2}\hat c\bd \pi+\gamma\int_{\tilde{\Omega}_1\times\tilde{\Omega}_2}\ext\Phi(\tfrac{\dpi}{\dhlambda})\dhlambda\,,%
		\]
		where we set the second integral to $+\infty$, if $\pi\not\ll\hlambda$,
		and
		\begin{align*}
			E_\gamma^{{\nu_1},{\nu_2}}[\pi]
			&=\begin{cases}
				F_\gamma[\pi]&\text{if }0\leq\pi\in\P(\tilde{\Omega}_1\times\tilde{\Omega}_2),\,\ppfw[i]{\pi}={\nu_i},\,i=1,2\,,\\
				\infty&\text{else,}
			\end{cases}
		\end{align*}
        for $\nu_i\in\Mp(\Omega_i)$, $i=1,2$.

		First, we state an auxiliary result ensuring that the marginal constraints are preserved by convolution. For simplicity, we state it for measures on $\R^{n}$ (but we could restrict everything to the respective domains $\Omega_{i}$, $\tilde \Omega_{i}$).
        Note that for $\nu\in\M(\Omega)$ the expression $\varphi_r * \nu$ can~--~thanks to the smoothness of $\varphi_r$~--~be interpreted both as smooth function or as measure in $\M$ with that smooth function as density.
        Here, we choose to interpret $\varphi_r * \nu$ as measure.
        A proof is given in \cref{appendix:nuconv:ppfw}.
		\begin{lemma}\label{thm:nuconv:ppfw}
			Let $\mu_1, \mu_2\in\P(\R^n)$ and let $\pi\in\P(\R^n\times\R^n)$ with $\ppfw{\pi} = \mu_i$, $i=1,2$.
			Let
			$\mu_i^\delta := \varphi_\delta*\mu_i,\,i=1,2$
            and
            $\pi_\delta:=G_{\delta}*\pi$.
			Then $\ppfw{\pi_\delta} = \mu_i^\delta$ for $i=1,2$.
		\end{lemma}

		\begin{theorem}[$\Gamma$-convergence for smoothed marginals] 
			\label{thm:conv_convergence}
			Let \cref{assmpt:general} hold and let $(\gamma , \delta) \phito 0$ denote
			\[
				\gamma\to0,\qquad \delta\to0,\qquad \gamma \pos{\Phi}\of{\frac{1}{{\delta}^{2n}}} \to0\,.
			\]
			Define the smoothed marginals as ${\mu_i^\delta}=\varphi_\delta*\tilde{\mu}_i$ for $i=1,2$.
			Then it holds that
			\begin{equation*}
				\Gammalim_{(\gamma , \delta)\phito 0} E_\gamma^{{\mu_1^\delta},{\mu_2^\delta}}=E_0^{{\mu_1},{\mu_2}}
			\end{equation*}
			with respect to weak-$*$ convergence in $\M({\Omega}_{1}\times{\Omega}_{2})$.
			Moreover, if ${\gamma},\delta\to0$ are chosen such that
			\[
				\gamma\big\lVert\tfrac{\bd\mu_1^\delta}{\dhlambda_1}\big\rVert_{\Lorl(\Omega_1,\dhlambda_1)}\to\infty
				\quad\text{or}\quad \gamma\big\lVert\tfrac{\bd\mu_2^\delta}{\dhlambda_2}\big\rVert_{\Lorl(\Omega_2,\dhlambda_2)}\to\infty\,,
			\]
			then $E^{{\mu_1^\delta},{\mu_2^\delta}}_\gamma$ does not have a finite ${\Gamma}$-limit. More precisely, even for feasible $\pi_\delta$ (i.e. with marginals $\mu_i^\delta$) it holds that
			\[
				\lim_{(\gamma , \delta)\phito 0} E^{{\mu_1^\delta},{\mu_2^\delta}}_\gamma[\pi_\delta] = \infty\,.
			\]
		\end{theorem}
		\begin{proof}
			This proof follows the outline of the proof in \cite[Theorem 5.1]{clason:2021}. 
			
			\begin{enumerate}[i)]
				\item%
				\emph{$\liminf$-condition:}
				Let $\pi_\delta\wsto{\pi}$, then $\lim_{\delta\to0}F_0[\pi_\delta] = F_0[\tilde{\pi}]$ due to $\hat c$ being continuous and bounded.
				Furthermore, 
				\[
					\int_{\tilde{\Omega}_1\times\tilde{\Omega}_2}\ext\Phi(\tfrac{\bd\check{\pi}}{\dhlambda})\dhlambda\geq-C\cdot\hlambda(\tilde{\Omega}_{1}\times\tilde{\Omega}_{2})
				\]
				for some $C>0$ only dependent on $\Phi$ for any $\check{\pi}\geq 0$.
				Thus,
				\begin{equation*}
					F_0[\tilde{\pi}]
					=\lim_{(\gamma , \delta)\phito 0}F_0[\pi_\delta]-\gamma C\cdot\hlambda\of{\tilde{\Omega}_{1}\times\tilde{\Omega}_{2}}
					\leq\liminf_{(\gamma , \delta)\phito 0}F_\gamma[\pi_\delta]\,.
				\end{equation*}
				Finally, the pushforward operator is weak-$*$ continuous, which implies that the marginal constraints are preserved under weak-$*$ convergence of $\pi_\delta$, ${\mu_1^\delta}$, and ${\mu_2^\delta}$ (note that $\mu_i^\delta \wsto \mu_i$).
				
				\item%
				\emph{$\limsup$-condition:}
				It suffices to consider a recovery sequence for ${\pi}\in\P({\Omega}_1\times{\Omega}_2)$, because  for ${\pi}\in\P(\tilde{\Omega}_1\times\tilde{\Omega}_2)\setminus\P({\Omega}_1\times{\Omega}_2)$ the marginal conditions for ${\mu_1}$ and ${\mu_2}$ can never be satisfied.
				
				If $E_0^{{\tilde\mu_1},{\tilde\mu_2}}[\tilde{\pi}]=\infty$, then $E_\gamma^{{\tilde\mu_1},{\tilde\mu_2}}[\tilde{\pi}]=\infty$ for every ${\gamma}$ and the $\limsup$ condition holds trivially. Let therefore $E_0^{{\tilde\mu_1},{\tilde\mu_2}}[\tilde{\pi}]$ be finite and
				set $\pi_\delta=G_{\delta}*\tilde{\pi}$.
				Then $0\leq\pi_\delta$ and $\pi_\delta\wsto\tilde{\pi}$ as well as $\ppfw[i]{\pi_\delta}={\mu_i^\delta}$ for $i=1,2$, by \cref{thm:nuconv:ppfw}.
				Finally, by Young’s convolution inequality,
				\[
					\tfrac{\bd\pi_\delta}{\dhlambda}\leq\norm{G_{\delta}}_{L^\infty(\Omega,\dhlambda)} \big\lVert\tfrac{\bd\tilde{\pi}}{\dhlambda}\big\rVert_{L^1(\Omega,\dhlambda)}\leq\frac{C}{{\delta}^{2n}}
				\]
				for some constant $C>0$.
				Thus,
				\begin{equation}\label{eq:conv_conv:lim_sup}
					\begin{split}
						\int_{\tilde{\Omega}_1\times\tilde{\Omega}_2}\ext\Phi(\tfrac{\dpi_\delta}{\dhlambda})\dhlambda
						&=\int_{\tilde{\Omega}_1\times\tilde{\Omega}_2}\Phi(\tfrac{\dpi_\delta}{\dhlambda})\dhlambda
						\leq \int_{\tilde{\Omega}_1\times\tilde{\Omega}_2}\pos{\Phi}(\tfrac{\dpi_\delta}{\dhlambda})\dhlambda\\
						&\leq \hlambda\of{\tilde{\Omega}_{1}\times\tilde{\Omega}_{2}} \gamma\pos{\Phi}\of{\frac{C}{{\delta}^{2n}}}\,,
					\end{split}
				\end{equation}
				and the right hand side vanishes for $(\gamma , \delta) \phito 0$ by the assumption on the (coupled) convergence of $\gamma$ and	$\delta$. Therefore,
				\begin{equation*}
					E_0^{{\tilde\mu_1},{\tilde\mu_2}}[\tilde{\pi}]
					=\lim_{(\gamma , \delta)\phito 0}\of{F_0[\pi_\delta] + \hlambda\of{\tilde{\Omega}_{1}\times\tilde{\Omega}_{2}}\gamma\pos{\Phi}\of{\frac{C}{{\delta}^{2n}}}\!}
					\geq\lim_{(\gamma , \delta)\phito 0}F_\gamma[\pi_\delta]\,.
				\end{equation*}
				
				\item%
				For the second statement, recall from \cref{thm:proj_contraction} that
				\[
					\begin{split}
					{\gamma}\big\lVert\tfrac{\bd\mu_1^\delta}{\dhlambda_1}\big\rVert_{\Lorl(\Omega_1,\dhlambda_1)} &\leq \gamma\max\of{1,\hlambda_2\of{\Omega_2}}\big\lVert\tfrac{\dpi_\delta}{\dhlambda}\big\rVert_{\Lorl(\Omega,\dhlambda)}\,,\\
					{\gamma}\big\lVert\tfrac{\bd\mu_2^\delta}{\dhlambda_2}\big\rVert_{\Lorl(\Omega_2,\dhlambda_2)} &\leq \gamma\max\of{1,\hlambda_1\of{\Omega_1}}\big\lVert\tfrac{\dpi_\delta}{\dhlambda}\big\rVert_{\Lorl(\Omega,\dhlambda)}\,.
					\end{split}
				\]
				By \cref{thm:lorl_norm_lower_est}, this immediately yields $F_\gamma[\pi_\delta]\to\infty$, and the assertion follows.\qedhere
			\end{enumerate}
		\end{proof}

		In the case $\lambda = \leb\times\leb$, the assumption $\gamma\pos{\Phi}\of{\frac{1}{{\delta}^{2n}}} \to0$
        is much stronger than necessary for some Young's functions.
        For example consider $\Phi(t) = \nicefrac{t^p}{p}$ with $p>1$ or $\Phi(t) = t\log t$. In those cases the condition
		\[
			{\gamma}{\delta^{2n}}\pos{\Phi}\of{\frac{1}{{\delta}^{2n}}} \to0
		\]
		suffices, as the following result states. For $\Phi(t) = t\log t$ this gives exactly the result in \cite[Theorem 5.1]{clason:2021}.
		
		\begin{corollary}\label{thm:conv_conv:monoton}
            Let $\lambda = \leb\times\leb$ and
			let $\Phi$ be a quasi-Young's function such that $t^{-1} \Phi(t)$ is monotone. Then it suffices to assume
			\[
				{\gamma}{\delta^{2n}}\pos{\Phi}\of{\frac{1}{{\delta}^{2n}}} \to0
			\]
			in \cref{thm:conv_convergence}.
		\end{corollary}
		\begin{proof}
			A refinement of estimate \eqref{eq:conv_conv:lim_sup} can be given. Using the monotonicity of $t^{-1} \Phi(t)$ we obtain
			\begin{align*}
				\int_{\tilde{\Omega}_1\times\tilde{\Omega}_2} \Phi(\tfrac{\dpi_\delta}{\dleb}) \dleb
				&\leq \int_{\tilde{\Omega}_1\times\tilde{\Omega}_2} \pos{\Phi}(\tfrac{\dpi_\delta}{\dleb}) \dleb\\
				& = \int_{\tilde{\Omega}_1\times\tilde{\Omega}_2} \tfrac{\dpi_\delta}{\dleb}\of{\tfrac{\dpi_\delta}{\dleb}}^{-1} \pos{\Phi} (\tfrac{\dpi_\delta}{\dleb}) \dleb\\
				&\leq \frac{\delta^{2n}}{C} \pos{\Phi}\of{\frac{C}{\delta^{2n}}} \int_{\tilde{\Omega}_1\times\tilde{\Omega}_2}\tfrac{\dpi_\delta}{\dleb} \dleb\\
				&=  \frac{\delta^{2n}}{C} \pos{\Phi}\of{\frac{C}{\delta^{2n}}}
			\end{align*}
			where again the right hand side vanishes by assumption. The assertion follows as in \cref{thm:conv_convergence}.
		\end{proof}		
	
	\subsection{Discretized problems}
	
		Here we describe a discretization of problem \eqref{eq:reg_kantorovich} and show two approximation results:
		\begin{enumerate}[i)]
			\item $\Gamma$-convergence of the discretized and regularized problem towards the continuous, \emph{regularized} problem \eqref{eq:reg_kantorovich}
			\item $\Gamma$-convergence of the discretized and regularized problem towards the continuous, \emph{unregularized} problem \eqref{eq:ot}
		\end{enumerate}
		
        We recall the problem data: Marginals $\mu_{i}\in\P(\Omega_{i})$ and finite measures $\lambda_{i}\in\Mp(\Omega_{i})$ with $\mu_{i}\ll\lambda_{i}$ for $i=1,2$ and a continuous and positive cost function $c$ on $\Omega$. We have $\lambda = \lambda_{1}\otimes\lambda_{2}$ and aim to discretize the problem
        \begin{equation}\label{eq:primal_ot_density}
          \begin{split}
            \min_{\pi\in\P(\Omega)}\int_{\Omega}c\tfrac{\dpi}{\dlambda} + \gamma\ext\Phi(\tfrac{\dpi}{\dlambda})\dlambda,\quad\text{s.t.}\quad \int_{\Omega_{2}}\tfrac{\dpi}{\dlambda}\dlambda_{2} & = \tfrac{\dmu_{1}}{\dlambda_{1}}\quad\text{$\lambda_1$-a.e}\\
            \int_{\Omega_{1}}\tfrac{\dpi}{\dlambda}\dlambda_{1} & = \tfrac{\dmu_{2}}{\dlambda_{2}}\quad\text{$\lambda_2$-a.e}.
          \end{split}
        \end{equation}
        We do a Galerkin discretization with piecewise constant functions. For $k\in\N$ let $(Q^{i,k}_j)$ be a sequence of finite partitions of $\Omega_i$ such that for every $j$ there is an $l$ with $Q_j^{i,k+1} \subset Q_l^{i,k}$ and such that $(Q^{1,k}_i)$, $(Q^{2,k}_j)$, and $I^k_{ij} := Q^{1,k}_i\times Q^{2,k}_j$ satisfy the following assumption:
                
		\begin{assumption}\label{asspt:disc}
			Let $A\subset\Omega$ be a Borel set and $\varepsilon>0$. Then there exists some $K\in\N$ such that for all $K\leq k\in\N$ the sets $A^k_+$, $A^k_-$ defined by
			\begin{align*}
					A^k_- &:= \bigcup_{\set{(i,j)}{I^k_{i,j}\subseteq A }} I^k_{i,j}\,,
					&A^k_+ &:= \bigcup_{\set{(i,j)}{I^k_{i,j}\cap A \neq \emptyset}} I^k_{i,j}
			\end{align*}
			satisfy
			\begin{align}\label{eq:asspt:disc}
				\nu(A^k_+) - \nu(A^k_-)&<\varepsilon\,,
			\end{align}
			for all $\nu\in\Mp(\Omega)$.
		\end{assumption}

		\begin{remark}
			If \cref{asspt:disc} is fulfilled for $\lambda = \lambda_{1}\otimes \lambda_{2}$, condition \eqref{eq:asspt:disc} holds analogously for $\lambda_1$ and $\lambda_2$, which can be seen as follows: For $A^i\subset\Omega_i$, $i=1,2$ let $A^{i,k}_+$ and $A^{i,k}_-$ be defined analogously based on $Q^{i,k}_j$. It holds
			\begin{align*}
			A^{1,k}_-\times\Omega_2 &= (A^1\times\Omega_2)^k_- \subset \Omega\\
			A^{1,k}_+\times\Omega_2 &= (A^1\times\Omega_2)^k_+ \subset \Omega\,.
			\end{align*}
			Thus,
			\begin{align*}
			\lambda_1(A^{1,k}_+) - \lambda_1(A^{1,k}_-)
			&= \ppfw[1]\lambda(A^{1,k}_+) - \ppfw[1]\lambda(A^{1,k}_-)
			= \lambda(A^{1,k}_+\times\Omega_2) - \lambda(A^{1,k}_-\times\Omega_2)\\
			&= \lambda((A^1\times\Omega_2)^k_+) - \lambda((A^1\times\Omega_2)^k_-) < \varepsilon
			\end{align*}
			by \eqref{eq:asspt:disc} and the argument holds analogously for $\lambda_2$.
		\end{remark}
	
		\Cref{asspt:disc} yields the following auxiliary result about piecewise constant approximation of measures.
		
		\begin{lemma}\label{thm:pwconst_wsto}
			Let \cref{asspt:disc} hold. Let $\nu\in\Mp(\Omega)$, $\nu^i\in\Mp(\Omega_i)$, $i=1,2$ and define
			\[
				\nu_k := \sum_{i,j} \frac{\nu(I^k_{i,j})}{\lambda(I^k_{i,j})} \1_{I^k_{i,j}}\,,
                \qquad
				\nu^i_k := \sum_{j} \frac{\nu^i(Q^{i,k}_j)}{\lambda_i(Q^{i,k}_j)} \1_{Q^{i,k}_j}\,,i=1,2\,.
            \]
			Then $\nu_k\lambda\wsto \nu$ and  $\nu^i_k\lambda_i \wsto\nu^i$.
		\end{lemma}

        A proof of \cref{thm:pwconst_wsto} can be found in \cref{appendix:disc}.
        We now set
        \begin{align*}
          \mu_{1,i}^{k} &:= \mu_{1}(Q^{1,k}_{i}), &\mu_{2,j}^{k} &:= \mu_{2}(Q^{2,k}_{j})\\
          \lambda_{1,i}^{k} &:= \lambda_{1}(Q^{1,k}_{i}), &\lambda_{2,j}^{k} &:= \lambda_{2}(Q^{2,k}_{j})\,.
        \end{align*}
        Then by~\cref{thm:pwconst_wsto}
        \begin{align*}
           \sum_{i} \tfrac{\mu_{1,i}^{k}}{\lambda_{1,i}^{k}}\1_{Q^{1,k}_{i}}\lambda_{1} & \wsto  \mu_{1} & \sum_{j} \tfrac{\mu_{2,j}^{k}}{\lambda_{2,j}^{k}}\1_{Q^{2,k}_{j}}\lambda_{2} & \wsto  \mu_{2}.
        \end{align*} 
        Note that division by zero is not a problem here, since $\lambda_1$ and $\lambda_2$ were assumed to have full support and hence $\lambda^k_{1,i},\lambda^k_{2,j}\neq 0$ for all $i,j,k$.
        We define the finite-dimensional spaces
        \[
              	\calV^{1,k} := \Span{\{\1_{Q^{1,k}_{i}}\mid i\}},\quad \calV^{2,k} := \Span{\{\1_{Q^{2,k}_{j}}\mid j\}},\quad \calV^{k} = \calV^{1,k}\otimes\calV^{2,k} = \Span{\{\1_{I^{k}_{ij}}\mid i,j\}}\,.
        \]
        
        The discrete problem is then one in $\calV^{k}$, namely
        \begin{equation}\tag{PD}\label{eq:reg_kantorivich_discretized_density}
          \inf_{\pi\in\calV^{k}}	\intO c \pi \dlambda+ \gamma \intO \ext\Phi(\pi)\dlambda,\quad
          \text{s.t.}\quad\parbox[t]{0.4\textwidth}{
            $\int_{\Omega_{2}}\pi\dlambda_{2} = \sum_{i}\tfrac{\mu^{k}_{1,i}}{\lambda^{k}_{1.i}}\1_{Q^{1,k}_{i}}$,\\
            $\int_{\Omega_{1}}\pi\dlambda_{1} = \sum_{j}\tfrac{\mu^{k}_{2,j}}{\lambda^{k}_{2.j}}\1_{Q^{2,k}_{j}}$
          }
        \end{equation}
                                      
        If we discretize the sought-after density $\tfrac{\dpi}{\dlambda}$ by
        \[
        \tfrac{\dpi}{\dlambda} = \sum_{ij}p_{ij}\1_{I^{k}_{ij}}\in \calV^{k}\,,
        \]
        we can derive the optimization problem for the unknown coefficients $p_{ij}$ as follows: The objective function is
        \[
        \int_{\Omega}c\tfrac{\dpi}{\dlambda} + \gamma\ext\Phi(\tfrac{\dpi}{\dlambda})\dlambda = \sum_{ij}\int_{I^{k}_{ij}}c\dlambda\; p_{ij} + \gamma\ext\Phi(p_{ij})\lambda^{k}_{1,i}\lambda^{k}_{2,j}.
        \]
        The first marginal constraint is given by
        \begin{align*}
          \int_{\Omega_{2}}\sum_{ij}p_{ij}\1_{Q^{1,k}_{i}}(x_{1})\1_{Q^{2,k}_{j}}(x_{2})\dlambda_{2}(x_{2}) = \sum_{i}\tfrac{\mu^{k}_{1,i}}{\lambda^{k}_{1,i}}\1_{Q^{1,k}_{i}}(x_{1})
        \end{align*}
        and this leads to the equation
        \[
        \sum_{j}p_{ij}\lambda^{k}_{2,j} = \tfrac{\mu^{k}_{1,i}}{\lambda^{k}_{1,i}}\,.
        \]
        A similar equation can be derived for the other marginal constraint.
        With
        \[
        c^{k}_{ij} := \tfrac1{\lambda^{k}_{1,i}\lambda^{k}_{2,j}}\int_{I^k_{ij}}c\dlambda\,,
        \]
        we arrive at the fully discretized problem
        \begin{equation*}
          \begin{split}
            \min_{p}\sum_{ij}\Big(c^{k}_{ij}p_{ij} +
            \gamma\ext\Phi(p_{ij})\Big)\lambda^{k}_{1,i}\lambda^{k}_{2,j},\quad \sum_{j}p_{ij}\lambda^{k}_{2,j} & = \tfrac{\mu^{k}_{1,i}}{\lambda^{k}_{1,i}}\\
             \sum_{i}p_{ij}\lambda^{k}_{1,i} &= \tfrac{\mu^{k}_{2,j}}{\lambda^{k}_{2,j}}.
          \end{split}
        \end{equation*}
        This is a finite-dimensional convex minimization problem with linear constraints. Several general methods could be used to solve this problem numerically, see, e.g.~\cite{blondel:2017,peyre:2019,lorenz:2019,lorenz2019orlicz}.
		The following theorem now guarantees that the discretized and regularized problem converges to the continuous regularized problem \eqref{eq:reg_kantorovich}.
        The proof makes use of the so-called Jensen's inequality (see e.g. \cite[Theorem 2.12.19]{bogachev:2007}), which we recall for the convenience of the reader:
        Let $K,\Omega\subset\R^d$, $M\subset K$ and $\nu\in\Mp(\R^d)$ with $\nu(\Omega) < \infty$.
        Let $\varphi: K \to \R^d$  be convex and let $f: \Omega \to M$ be a $\nu$-integrable function such that $\varphi(f)$ is $\nu$-integrable. Then,
        \[
            \nu(\Omega)\cdot\varphi\Big(\frac{1}{\nu(\Omega)}\intO f\dnu\Big) \leq \intO \varphi(f)\dnu\,.
        \]
		
		\begin{theorem}\label{thm:disc:disc_conv}
			Let \cref{assmpt:general,asspt:disc} hold.
			Let $c\in\CC(\Omega,[0,\infty))$, 
			$\Phi$ be a quasi-Young's function,
			and let $\gamma>0$.	Then the minimization problem \eqref{eq:reg_kantorivich_discretized_density}
			$\Gamma$-converges to problem~\eqref{eq:primal_ot_density} w.r.t. weak-$*$ convergence in $\M(\Omega)$ as $k\to\infty$.
		\end{theorem}
		\begin{proof}
            We define, $F,F_k:\M(\Omega) \to \R\cup\{\infty\}$ via
			\begin{equation*}
				F_k(\pi) =
					\begin{cases}
						\intO c \dpi+ \gamma \intO \ext\Phi(\tfrac{\dpi}{\dlambda})\dlambda,
                            & \tfrac{\bd\ppfw[1]\pi}{\dlambda_1} = \sum_{i}\tfrac{\mu^{k}_{1,i}}{\lambda^{k}_{1,i}}\1_{Q^{1,k}_{i}}\,, \tfrac{\bd\ppfw[2]\pi}{\dlambda_2} = \sum_{j}\tfrac{\mu^{k}_{2,j}}{\lambda^{k}_{2,j}}\1_{Q^{2,k}_{j}}\,,\\
                            & 0\leq\pi\ll\lambda, \tfrac{\dpi}{\dlambda}\in\calV_{k}\\
                        \infty,
                            &\text{else}
					\end{cases}
			\end{equation*}
			and
			\[
				F(\pi) =
					\begin{cases}
						\intO c\dpi + \gamma\intO \ext\Phi(\tfrac{\dpi}{\dlambda}) \dlambda,&
						0\leq\pi\ll\lambda,\, \ppfw[i]\pi = \mu_i\,,i=1,2\\
						\infty,&\text{else.}
					\end{cases}
			\]
			Given an arbitrary $\pi\in\M(\Omega)$, we now check the two conditions for $\Gamma$-convergence.
			
			\begin{enumerate}[i)]
				\item%
				\emph{$\liminf$-condition:}
				Let $(\pi_k)\subset\M(\Omega)$ such that $\pi_k \wsto \pi$.
				
				If $F(\pi) < \infty$, pass to a subsequence (denoted by the same symbol) with finite values $F_k(\pi_k)$. 
				Because $\pi_k \wsto \pi$, we have
				\[
					\intO c \dpi_k \to \intO c\dpi\,,
				\]
				and moreover, $\liminf_{k\to\infty}\intO \ext\Phi(\tfrac{\dpi_k}{\dlambda}) \dlambda \geq \intO \ext\Phi(\tfrac{\dpi}{\dlambda}) \dlambda$ by \cref{thm:ws_liminf}.
				
				If $F(\pi) = \infty$, assume for a contradiction that $\liminf_{k\to\infty} F_k(\pi_k) < \infty$. Pass to a subsequence (not renamed) $(\pi_{k})$ with $\lim_{k\to\infty}F_{k}(\pi_{k}) = \liminf_{k\to\infty} F_k(\pi_k)$. 
				
				Since, $\pi_k \wsto \pi$ and, by \cref{thm:pwconst_wsto}, $\sum_{i}\tfrac{\mu^{k}_{1,i}}{\lambda^{k}_{1,i}}\1_{Q^{1,k}_{i}} \wsto \tfrac{\dmu_1}{\dlambda_{1}}$, and $\sum_{j}\tfrac{\mu^{k}_{2,j}}{\lambda^{k}_{2,j}}\1_{Q^{2,k}_{j}} \wsto \tfrac{\dmu_2}{\dlambda_{2}}$, we see that $\tfrac{\dpi}{\dlambda}$ always satisfies the marginal constraints and positivity. Hence, $F(\pi)=\infty$ can not occur due to violation of these constraints.
				
				So, if $\pi$ satisfies the marginals but $\pi\not\ll\lambda$ or $\pi\ll\lambda$ with $\intO \ext\Phi(\tfrac{\dpi}{\dlambda}) \dlambda = \infty$, then it holds $\liminf_{k\to\infty}F_k(\pi_k) = \infty$ by \cref{thm:ws_liminf}, which again is a contradiction.
			
				Thus, $\liminf F_{k}(\pi_{k}) \geq F(\pi)$ in every case.
				
				\item%
				\emph{$\limsup$-condition:}
				If $F(\pi)=\infty$, the condition $\limsup_{k\to\infty}F_{k}(\pi_{k})$ is trivially fulfilled for $\pi_k := \pi$.
                                Hence, consider $F(\pi) < \infty$ and define $(\pi_k)$ by
				\[
					\frac{\dpi_k}{\dlambda} := \sum_{i,j} \frac{\pi(I^k_{ij})}{\lambda(I^k_{ij})}\1_{I_{ij}^{k}}\,.
				\]
				Note that by assumption $\lambda(I^k_{ij})>0$ and therefore $\pi_k \geq 0$ by definition.
				One easily sees that 
				\[
					\intO[1] \frac{\dpi_k}{\dlambda} \dlambda_1 = \sum_{j}\frac{\mu^{k}_{2,j}}{\lambda^{k}_{2,j}}\1_{Q^{2,k}_{j}},\quad\text{and}\quad \intO[2] \frac{\dpi_k}{\dlambda} \dlambda_2 = \sum_{i}\frac{\mu^{k}_{1,i}}{\lambda^{k}_{1,i}}\1_{Q^{1,k}_{i}},
				\]
				i.e., $F_k(\pi_k) < \infty$.
				In particular, $\pi_k \wsto \pi$ by \cref{thm:pwconst_wsto}.
				
				It remains to show $\limsup_{k\to\infty}F_k(\pi_k) \leq F(\pi)$. As before, we have
				\[
					\intO c \dpi_k \to \intO c \dpi\,.
				\]
				Moreover, we get from Jensen's inequality that
                \[
                    \Phi\Big(\tfrac{\pi(I^{k}_{ij})}{\lambda(I^{k}_{ij})}\Big) = \Phi\Big(\tfrac1{\lambda(I^{k}_{ij})}\int_{I^{k}_{ij}}\tfrac{\dpi}{\dlambda}\dlambda\Big)\leq \tfrac1{\lambda(I^{k}_{ij})}\int_{I^{k}_{ij}}\Phi(\tfrac{\dpi}{\dlambda})\dlambda.
                \]
                With this we obtain
				\begin{align*}
                	\intO\Phi(\tfrac{\dpi_{k}}{\dlambda})\dlambda
                	& = \intO \Phi\of{\sum_{ij}\tfrac{\pi(I^{k}_{ij})}{\lambda(I^{k}_{ij})}\1_{I^{k}_{ij}}}\dlambda
                	= \sum_{i,j} \int_{I^k_{ij}}\dlambda\, \Phi\of{\tfrac{\pi(I^{k}_{ij})}{\lambda(I^{k}_{ij})}}
                	= \sum_{i,j} \lambda(I^{k}_{ij}) \Phi\of{\tfrac{\pi(I^{k}_{ij})}{\lambda(I^{k}_{ij})}}\\
                    & \leq \sum_{ij}\int_{I^{k}_{ij}}\Phi(\tfrac{\dpi}{\dlambda})\dlambda = \intO \Phi(\tfrac{\dpi}{\dlambda})\dlambda
				\end{align*}
                                which shows the $\limsup$-condition.\qedhere
			\end{enumerate}
		\end{proof}
		
		$\Gamma$-convergence for \emph{simultaneously} decreasing the regularization parameter \emph{and} refining the discretization proves to be a harder problem. The following example shows that the convergence rate of $\gamma_k$ must be linked to the convergence rate of the discretization.

		\begin{example}\label{expl:conv}
			Let $\Omega =[0,1]^2$, ${\mu_1} = {\mu_2} = \delta_0$ the Dirac measure at zero and $\lambda_i = \leb_i$ the Lebesgue measure on $\Omega_i$. Then clearly $\pi = \delta_0$ is the only feasible and thus the optimal transport plan. Let $0<h_{k}$ be a sequence, which is to be chosen.
                        
			We consider $\Phi(t) = t\log t$ and have the discretized optimal plan
			\[
                        \pi_k = \frac{\pi([0,h_k]^2)}{\lambda([0,h_k]^2)} \1_{{[0,h_k]}^2} = \frac{ \1_{{[0,h_k]}^2}}{h_k^2}\,.
			\]
			However, it holds that
			\[
                        \gamma_k \intO \Phi(\pi_k) \dlambda = \gamma_k \log(h_k^{-2}),
			\]
                        and hence, the $\limsup$ condition can not be fulfilled if $h_{k}$ is such that $\gamma_{k}\log(h_{k}^{-2})\not\to 0$ (which holds, e.g., for $h_{k} = \exp(-1/\gamma_{k})$).
                        
            Let now $\tilde\lambda_i = \delta_0$ be the Dirac measure at zero. The discretized optimal plan now is
            \[
            	 \tilde\pi_k = \frac{\pi([0,h_k]^2)}{\tilde\lambda([0,h_k]^2)} \1_{{[0,h_k]}^2} = \1_{{[0,h_k]}^2}\,, \qquad\text{with } \tilde\lambda := \tilde\lambda_1\otimes\tilde\lambda_2
            \]
            and hence,
            \[
            	\gamma_k \intO \Phi(\tilde\pi_k)\bd\tilde\lambda = \gamma_k \Phi(\tilde\pi_k(0)) = \gamma_k \Phi(1) = 0\,.
            \]
            Thus, $h_k$ may be chosen arbitrary in this case.
		\end{example}
		
		Using the insights of \cref{expl:conv} the desired result can be formulated.
		
		\begin{theorem}\label{thm:disc:conv}
			Under the assumptions of \cref{thm:disc:disc_conv}, let $(h_k)$ be a positive sequence  with
			$h_k \leq \lambda(I^k_{ij})$ for all $i,j$
			and let $(\gamma_k)$ be a sequence converging to zero such that
			\begin{equation}\label{cond:phiplus}
				\gamma_k \pos{\Phi}(h_k^{-1}) \to 0\,.
			\end{equation}
			Then the minimization problem
			\begin{align*}
				\inf \left\{\intO\right. c \dpi &+ \gamma_k \left.\intO \ext\Phi(\tfrac{\dpi}{\dlambda})\dlambda \,\middle|\right.\,
                \pi\ll\lambda\,,0\leq\tfrac{\dpi}{\dlambda}\in\calV^{k},\\
				&\left.\tfrac{\bd\ppfw[1]\pi}{\dlambda_1} = \sum_{i}\tfrac{\mu^{k}_{1,i}}{\lambda^{k}_{1,i}}\1_{Q^{1,k}_{i}},\, \tfrac{\bd\ppfw[2]\pi}{\dlambda_2} = \sum_{j}\tfrac{\mu^{k}_{2,j}}{\lambda^{k}_{2,j}}\1_{Q^{2,k}_{j}}\right\}
			\end{align*}%
			$\Gamma$-converges to
			\eqref{eq:ot}
			w.r.t. weak-$*$ convergence in $\M(\Omega)$ as $k\to\infty$.		
		\end{theorem}
		\begin{proof}
			Let $F_k$ be defined as in the proof of \cref{thm:disc:disc_conv} with $\gamma = \gamma_k$ and let $F:\M(\Omega) \to \R\cup\{\infty\}$ be defined via
			\[
				F(\pi) =
					\begin{cases}
						\intO c\dpi,&
						0\leq \pi\in\P(\Omega),\,\ppfw {\pi}  = \mu_i,\,i=1,2,\\
						\infty,&\text{else.}
					\end{cases}
			\]
			
			Given an arbitrary $\pi\in\M(\Omega)$, check now the two conditions for $\Gamma$-convergence.
			
			\begin{enumerate}[i)]
				\item%
				\emph{$\liminf$-condition:}
				Let $\pi$ be any measure and $(\pi_k)$ be such that $\pi_k \wsto \pi$.
				
				The case $F(\pi)= \infty$ can be treated similarly to the proof of \cref{thm:disc:disc_conv}, with the difference that we don't have to consider the case $\pi\not\ll\lambda$.
                For $F(\pi) < \infty$ we get
				\[
					\intO c \dpi_k \to \intO c \dpi\,.
				\]
                Moreover, since $\Phi$ is bounded from below, we can extract a subsequence such that $\intO \Phi(\tfrac{\dpi_{k}}{\dlambda})\dlambda$ is bounded from above and below and obtain
				\[
					\lim_{k\to\infty} \gamma_k \intO \Phi(\tfrac{\dpi_{k}}{\dlambda}) \dlambda = 0
				\]
				which proves $\liminf F_{k}(\pi_{k}) \geq F(\pi)$.
				
				\item%
				\emph{$\limsup$-condition:}
				For $F(\pi) = \infty$ we have nothing to prove.
				
				If $F(\pi) < \infty$, define $\pi_k$ as in the proof of \cref{thm:disc:disc_conv}.
				Then, the marginal constraints are satisfied and $\pi_k \wsto \pi$. Hence $\intO c\dpi_{k} \to \intO c\dpi$ and it remains to show that
				\[
					\limsup_{k\to\infty} \gamma_k \intO \Phi(\tfrac{\dpi_k}{\dlambda}) \dlambda \leq 0\,.
				\]
				Using $0\leq\pi(I^{k}_{ij}) \leq \pi(\Omega) = 1$ and monotonicity of $\Phi_{+}$ we get
				\begin{equation}\label{eq:disc:conv:est}
					\begin{split}
						\intO \ext\Phi(\tfrac{\dpi_k}{\dlambda}) \dlambda 
						& = \sum_{ij} \Phi\of{\tfrac{\pi(I^k_{ij})}{\lambda(I^k_{ij})}} \lambda(I^k_{ij})
						\leq \sum_{ij} \pos{\Phi}\of{\tfrac{\pi(I^k_{ij})}{\lambda(I^k_{ij})}} \lambda(I^k_{ij})\\
						&\leq \sum_{ij} \pos{\Phi}\Big(\tfrac{1}{\lambda(I^k_{ij})}\Big) \lambda(I^k_{ij})
						\leq \pos{\Phi}(h_k^{-1}) \sum_{ij}\lambda(I^{k}_{ij})\\
						&= \pos{\Phi}(h_k^{-1})\lambda(\Omega)\,.
					\end{split}
				\end{equation}
				Hence,
				\begin{align*}
				\gamma_k \intO \Phi(\tfrac{\dpi_k}{\dlambda}) \dlambda \leq\lambda(\Omega) \gamma_k \pos{\Phi}(h_k^{-1}) \to0
				\end{align*}
                                as desired.\qedhere
			\end{enumerate}
		\end{proof}
		
		\begin{corollary}
			\cref{thm:disc:disc_conv,thm:disc:conv} remain true if instead of a continuous cost function $c$, a sequence of step functions $(c_k)$, constant on the sets $I^k_{ij}$ and converging to $c$ uniformly, is used.
		\end{corollary}
		\begin{proof}
			Rewrite the integral $\intO c_k \dpi_k$ as the dual pairing
			 $\sca{c_k}{\pi_k}$, where $(c_k) \subset \CCb(\Omega)$ converges strongly and $\pi_k \subset\M(\Omega)$ converges weakly-$*$. Because the variation norm $\norm{\pi_k}$ is bounded (\cite[Corollary 2.6.10]{megginson:1998}), $\sca{c_k}{\pi_k}\to\sca c\pi$ can be seen by standard arguments, which yields the assertion.
		\end{proof}
		
		Similarly to \cref{thm:conv_conv:monoton}, for $t^{-1} \Phi(t)$ monotone the assumption on $(\gamma_k)$ can be weakened.
		
		\begin{corollary}
			Let $\Phi$ be a quasi-Young's function such that $t^{-1} \Phi(t)$ is monotone. Then it suffices to assume
			\[
				\gamma_k {h_k}\pos{\Phi}\of{h_k^{-1}} \to 0
			\]
			instead of condition \eqref{cond:phiplus} in \cref{thm:disc:conv}.
		\end{corollary}
		\begin{proof}
			Using the monotonicity of $t^{-1} \Phi(t)$, \eqref{eq:disc:conv:est} can be refined to
			\begin{align*}
				\intO \ext\Phi (\tfrac{\dpi_k}{\dlambda}) \dlambda &= \sum_{ij} \Phi\of{\tfrac{\pi(I^k_{ij})}{\lambda(I^k_{ij})}} \tfrac{\lambda(I^k_{ij})}{\pi(I^k_{ij})} \tfrac{\pi(I^k_{ij})}{\lambda(I^k_{ij})}  \lambda(I^k_{ij})
				\leq \sum_{ij} \Phi_+\of{\tfrac{\pi(I^k_{ij})}{\lambda(I^k_{ij})}} \tfrac{\lambda(I^k_{ij})}{\pi(I^k_{ij})} \pi(I^k_{ij})\\
				& \leq \sum_{ij} \Phi_+\of{\of{\lambda(I^k_{ij})}^{-1}} \lambda(I^k_{ij}) \pi(I^k_{ij})
				\leq  h_k  \Phi_+(h_k^{-1})\sum_{ij} \pi(I^k_{ij})\\
				&= h_k\Phi_+(h_k^{-1}) \pi(\Omega) =  h_k\Phi_+(h_k^{-1})
			\end{align*}
			and the assertion follows as in \cref{thm:disc:conv}.
		\end{proof}
	
	\section{Conclusion}
	
	Employing regularization in Orlicz spaces to the optimal transport problem allows to generalize the existence results of \cite{clason:2021,lorenz:2019} for both the primal and the predual problem and under mild assumptions, the results hold for regularization w.r.t. product measures $\lambda = \lambda_1 \otimes \lambda_2$. More precisely, primal solutions exist if and only if the marginals are functions in the appropriate Orlicz spaces and existence of optimizers in $L^q$ for the predual problem has been shown for the special case $\Phi(t) = \nicefrac{t^p}p$, $p\geq 2$.
	
	A combined regularization and smoothing approach leads to a family of well-posed approximations that $\Gamma$-converge to the original Kantorovich formulation if the regularization and smoothing parameters are coupled in an appropriate way. This gives a generalization of the corresponding result \cite[Theorem 5.1]{clason:2021} for $\Phi(t) = t\log t$.
	Similarly, a combined regularization and discretization approach leads to another family of approximations. It could be proven that, again, $\Gamma$-convergence is guaranteed if the regularization parameter and the discretization fineness are coupled in an appropriate way.
	
	Existence of solutions of the dual problem for general (quasi-)Young's functions has been considered in a different framework in \cite{Leonard:2008}. Still, future work might investigate, if the result can also be achieved by the approach considered here.
	Moreover, numerical methods for solving the regularized problem \eqref{eq:reg_kantorovich} have not been discussed here.

\section*{Acknowledgments}
    The authors would like to thank two anonymous reviewers for their suggestions regarding the presentation and valuable comments, which also helped to close some gaps in the argumentation.

\section*{Appendix}
\appendix

\section{Proof of \cref{thm:lorl:eq_yf}}\label{appendix:thm:lorl:eq_yf}

\begin{proof}
    Let $\Phi(t) := \int_0^t \varphi(s) \ds$ be as in \cref{def:yf}. Then with 
    $\tilde{\varphi} := \varphi\cdot \1_{(t_0,\infty)}$
	it holds that $\tilde{\Phi}(t) = \int_0^t \tilde{\varphi}(s) \ds$ and hence, 	$\tilde{\Phi}$ is a Young's function.

	To show $\Lorl(\Omega, \dnu)=\Lorl[\tilde{\Phi}](\Omega, \dnu)$ it suffices to show that for any $f$
	\[
		\norm{f}_{\Lorl(\Omega, \dnu)} < \infty \Leftrightarrow \norm{f}_{\Lorl[\tilde{\Phi}](\Omega, \dnu)} < \infty\,.
	\]
	
	If $\norm{f}_{\Lorl(\Omega, \dnu)}<\infty$, then $\norm{f}_{\Lorl[\tilde{\Phi}](\Omega, \dnu)}<\infty$ trivially. Let therefore $\norm{f}_{\Lorl[\tilde{\Phi}](\Omega, \dnu)}~<~\infty$, i.e. there is a $\gamma$ such that
	\[
		\intO \tilde{\Phi}\left(\frac{\abs{f(x)}}{\gamma}\right) \dnu(x)= \int_{\Omega_\gamma(f)} \Phi\left(\frac{\abs{f(x)}}{\gamma}\right) \dnu(x) - \nu(\Omega_\gamma(f)) \Phi(t_0) \leq 1\,,
	\]
	where $\Omega_\gamma(f) := \set{x}{\frac{\abs{f(x)}}{\gamma}\geq t_0}$. Let $r\in(0,1]$ and write
    \[
		\intO \Phi\left(r\frac{\abs{f(x)}}{\gamma}\right) \dnu(x)= \left(
		\int_{\set{x}{\frac{\abs{f(x)}}{\gamma}\leq t_0}} +
		\int_{\set{x}{t_0 \leq \frac{\abs{f(x)}}{\gamma}\leq \frac{t_0}{r}}}+
		\int_{\set{x}{\frac{t_0}{r} \leq \frac{\abs{f(x)}}{\gamma}}} \right)
		\Phi\left(r\frac{\abs{f(x)}}{\gamma}\right) \dnu(x)\,.
    \]
	Since $\Phi(0) = 0$ and $\Phi$ is convex, for every $s\in[0,1]$ it holds that $\Phi(s x) \leq s\Phi(x)$ and we obtain an upper bound for the first integral by
	\[
		\int_{\set{x}{\frac{\abs{f(x)}}{\gamma}\leq t_0}}\Phi\left(r\frac{\abs{f(x)}}{\gamma}\right) \dnu(x) \leq
		r \int_{\set{x}{\frac{\abs{f(x)}}{\gamma}\leq t_0}}\Phi\left(\frac{\abs{f(x)}}{\gamma}\right) \dnu(x) \leq
		r C
	\]
	for some constant $C$. With the same argument,
	\begin{align*}
		\int_{\set{x}{\frac{t_0}{r} \leq \frac{\abs{f(x)}}{\gamma}}}\Phi\left(r\frac{\abs{f(x)}}{\gamma}\right) \dnu(x) 
        &\leq r \int_{\set{x}{\frac{t_0}{r} \leq \frac{\abs{f(x)}}{\gamma}}}\Phi\left(\frac{\abs{f(x)}}{\gamma}\right)\dnu(x)\\
        &\leq r \int_{\set{x}{t_0 \leq \frac{\abs{f(x)}}{\gamma}}}\Phi\left(\frac{\abs{f(x)}}{\gamma}\right)\dnu(x)\\
        &= r \left(\intO \tilde{\Phi}\left(\frac{\abs{f(x)}}{\gamma}\right) \dnu(x) + \nu(\Omega_\gamma(f))\Phi(t_0)\right)\\
        &\leq r \left(1 + \nu(\Omega_\gamma(f))\Phi(t_0)\right)\,.
	\end{align*}
	Finally, since $\set{x}{t_0 \leq \frac{\abs{f(x)}}{\gamma}\leq \frac{t_0}{r}} \subseteq \set{x}{t_0 \leq \frac{\abs{f(x)}}{\gamma}}$, the same estimate holds for the second integral. Combining the estimates for the three integrals one obtains
	\[
		\intO \Phi\left(r\frac{\abs{f(x)}}{\gamma}\right) \dnu(x) \leq r\left( C + 2\cdot(1 + \nu(\Omega_\gamma(f))\Phi(t_0))\right)\,,
	\]
	which is less or equal than $1$ for $r$ small enough.
\end{proof}

\section{Proof of \cref{thm:nuconv:ppfw}}\label{appendix:nuconv:ppfw}

\begin{proof}
	We only treat the case $i=1$ (the case $i=2$ being analogous).
    Recalling $G_{\delta} = \varphi_{\delta}\otimes\varphi_{\delta}$ and using Fubini's Theorem we get
	\begin{align*}
    \frac{\dmu_1^\delta}{\dleb}(y_1)
        &= \frac{\bd\of{\varphi_\delta * \ppfw[1]{\pi}}}{\dleb}(y_1)
		= \int_{\R^n} \varphi_\delta(y_1-x_1)\bd(\ppfw[1]{\pi})(x_1)\\
		&= \int_{\R^n\times\R^n} \varphi_\delta(y_1-x_1) \dpi(x_1,x_2)
		= \int_{\R^n} \int_{\R^n\times\R^n} G_\delta(y_1-x_1, y_2-x_2) \dpi(x_1, x_2) \bd y_2\\
		&= \int_{\R^n} \of{G_\delta * \pi}(y_1, y_2) \bd y_2
		= \frac{\bd\ppfw[1]\pi_\delta}{\dleb} (y_1)\,,
	\end{align*}
    where the third equality follows directly from the definition of the Lebesgue integral via simple functions.
  \end{proof}

\section{Proof of \cref{thm:pwconst_wsto}}\label{appendix:disc}

\begin{proof}
	Let $A\subset\Omega$ be a Borel set, $\varepsilon>0$ and $A^k_-$, $A^k_+$ as in \cref{asspt:disc}. Then for all $i,j$ it holds
	\[
		\lambda(I^k_{i,j}\cap A^k_+) =
			\begin{cases}
				\lambda(I^k_{ij}),&I^k_{ij}\cap A \neq \emptyset,\\
				0,&\text{else}
			\end{cases}
	\]
	and similarly for $\lambda(I^k_{i,j}\cap A^k_-)$. In combination with \eqref{eq:asspt:disc} this yields
	\begin{align*}
		\int_A \nu_k \dlambda - \nu(A)
		&\leq \int_{A^k_+} \nu_k\dlambda - \nu(A^k_-)
		= \sum_{i,j} \frac{\nu(I^k_{ij})}{\lambda(I^k_{ij})}\lambda(I^k_{i,j}\cap A^k_+) - \nu(A^k_-)\\
		&= \sum_{\set{(i,j)}{I^k_{i,j}\cap A \neq \emptyset}} \nu(I^k_{i,j}) - \nu(A^k_-)
		= \nu(A^k_+) - \nu(A^k_-)
		< \varepsilon\,,
	\end{align*}
	for $k$ large enough.
	Using an analogous argument for a lower bound, we get
	\[
		\left|\int_A \nu_k \dlambda - \nu(A)\right| < \varepsilon\,,
	\]
	which yields the first assertion. Analogous argumentation proves the result for $\nu^i\in\Mp(\Omega_i)$.
\end{proof}

	\bibliographystyle{abbrvurl}
	\bibliography{./literature}

\end{document}